\documentclass{ejm}
\usepackage{verbatim}
\usepackage{latexsym}
\usepackage{amsmath,amsfonts}
\usepackage{eucal}
\usepackage{cases}
\usepackage{cite}

\usepackage{mathtext}
\usepackage[cp1251]{inputenc}
\usepackage[T2A]{fontenc}
\usepackage[dvips]{graphicx}
\usepackage{amsmath}
\usepackage{amssymb}
\usepackage{amsxtra}
\usepackage{latexsym}
\usepackage{ifthen}
\usepackage{color}
\usepackage{cite}
\usepackage{hyperref}

\newtheorem{theorem}{Theorem}
\newtheorem{definition}{Definition}[section]
\newtheorem{lemma}{Lemma}[section]

\numberwithin{equation}{section}

\newcommand{\red}[1]{\ifmmode\mathbf{\textcolor{red}{#1}}\else \textbf{\textcolor{red}{#1}}\fi}

\newcommand{\blue}[1]{\ifmmode\mathbf{\textcolor{blue}{#1}}\else \textbf{\textcolor{blue}{#1}}\fi}

\begin{document}

\title{On travelling wave solutions of a model of a liquid film flowing down a fibre}

\shorttitle{Travelling waves of a fibre coating model}

\author[H. Ji, R. Taranets, and M. Chugunova]{H\ls A\ls N\ls G\ls J\ls I\ls E\ls \ns  J\ls I\ls $^1$, R\ls O\ls M\ls A\ls N\ls \ns T\ls A\ls R\ls A\ls N\ls E\ls T\ls S $^2$ and M\ls A\ls R\ls I\ls N\ls A\ls \ns  C\ls H\ls U\ls G\ls U\ls N\ls O\ls V\ls A $^3$}

\affiliation{%
$^1\,$Department of Mathematics, University of California Los Angeles,
Los Angeles, CA 90095, USA\\
 email\textup{\nocorr: \texttt{hangjie@math.ucla.edu}}\\
$^2\,$Institute of Applied Mathematics and Mechanics of the NASU, 
Dobrovol'skogo Str., 84100, Sloviansk, Ukraine\\
email\textup{\nocorr: \texttt{taranets\_r@yahoo.com}}\\
$^3\,$Claremont Graduate University, 150 E. 10th Str., Claremont, California 91711, USA\\
email\textup{\nocorr: \texttt{marina.chugunova@cgu.edu}}\\
}

\maketitle

\begin{abstract}
Existence of non-negative weak solutions is shown for a full curvature thin-film model of
a liquid thin film flowing down a vertical fibre. The proof is based on the application of
a priori estimates derived for energy-entropy functionals. Long-time behaviour of these weak
solutions is analysed and, under some additional constraints
for the model parameters and initial values, convergence towards a travelling wave solution is obtained.
Numerical studies of energy minimizers and travelling waves are presented to illustrate analytical results.
\end{abstract}

\vspace{2pc}

%

\section{ Introduction }
Unstable thin viscous liquid films coating a vertical fibre display complex interfacial dynamics and various flow regimes. Driven by Rayleigh-Plateau instability and gravity effects, a rich variety of dynamics can occur including the formation of droplets, regular travelling wave patterns, and irregular droplet coalescence.
Such dynamics has attracted a lot of attention from researchers in recent years due to its widespread applications in heat and mass exchangers\cite{zeng2018thermohydraulic}, desalination \cite{sadeghpour2019water,zeng2019highly}, and particle capturing systems \cite{Sadeghpour2020Experimental}.

With proper choices of flow rate, liquid, fibre radius, and inlet geometry, three typical flow regimes are observed in experiments \cite{kalliadasis1994drop,Ji2020Modeling,sadeghpour2017effects}:  a convective instability regime where bead coalescence happens repeatedly, a travelling wave regime where a steady strain of beads flow down the fibre at a constant speed, and an isolated droplet regime where widely spaced large droplets are separated by small wave patterns. When other system parameters are fixed, varying the flow rate from high to low can lead to a flow regime transition from the convective to the travelling wave regime, and eventually to the isolated droplet regime.
A better understanding of the travelling wave pattern is expected to provide insights for many engineering applications that utilize steady trains of liquid beads.

At small flow rates, the dynamics of the axisymmetric flow on a cylinder is typically modelled by classical lubrication theory. Under the assumption that the film thickness is much smaller than the radius of the cylinder, Frenkel \cite{frenkel1992nonlinear} proposed a weakly nonlinear thin-film equation for the film thickness $h$ (or the height of the film) that captures both stabilizing and destabilizing effects of the surface tension in the dynamics. This evolution equation was further investigated by Kalliadasis \& Chang \cite{kalliadasis1994drop},
Chang \& Demekhin \cite{chang1999mechanism}, and Marzuola, Swygert \& Taranets  \cite{marzuola2019}. Similar models for weakly rippled thin films were also studied in \cite{shlang1982irregular, rosenau1989evolution}, and the paper of Rosenau and Oron \cite{rosenau1989evolution} incorporates fully nonlinear curvature terms that account for the deformation of the film interface based on asymptotic analysis.
To relax the thin film assumption, Craster \& Matar \cite{craster2006viscous} developed an asymptotic model which assumes that the film thickness is much smaller than the capillary length. 
In 2000, Kliakhandler et al.~\cite{KDB} extended the thin film model to consider a thick layer of viscous fluid by introducing fully nonlinear curvature terms, which leads to an evolution equation \eqref{r-00} for the film thickness $h(x,t)$,
\begin{equation}\label{r-00}
h_{t} + \frac{1}{h+r_0}\left[ Q(h) \left( 1  + \sigma^{-1}  \left[ \frac{h_{xx}}{(1+ h_x^2)^{3/2}} - \frac{1}{(h+r_0)(1+ h_x^2)^{1/2}} \right]_x \right)\right]_x  = 0,
\end{equation}
where $\sigma > 0$ is the Bond number, $r_0 > 0$ is the dimensionless fibre radius,
and the mobility $Q(h)$ takes the form
\begin{equation}
    Q(h) = \tfrac{1}{16} \Bigl[ 4 (h + r_0)^4 \log \bigl( \tfrac{h+r_0}{r_0}\bigr) - h (3 h^3 + 12r_0 h^2 + 14r_0^2 h + 4r_0^3)\Bigr].
\end{equation}
While equation \eqref{r-00} is a model equation that was not rigorously derived asymptotically, it constitutes all the necessary terms to describe the corresponding physical process.
The term $[Q(h)]_x$ in \eqref{r-00} represents gravitational effects,  ${h_{xx}}/{(1+ h_x^2)^{3/2}}$ and ${(h+r_0)^{-1}(1+ h_x^2)^{-1/2}}$ describe the stabilizing and destabilizing roles of the surface tension due to axial and azimuthal curvatures of the interface, respectively. 

In 2019, Ji et al.~\cite{Bert2019} studied a family of full lubrication models that incorporate slip boundary conditions, fully nonlinear curvature terms, and a film stabilization mechanism.
The film stabilization term,
\begin{equation}
    \Pi(h) = -\frac{A}{h^3}, \quad A>0,
\label{pi}
\end{equation}
is motivated by the form of disjoining pressure widely used in lubrication equations \cite{thiele2011depinning,RevModPhys.57.827} to describe the wetting behaviour of a liquid on a solid substrate, and the scaling parameter $A > 0$ is typically selected based on a stable liquid layer in the coating film dynamics.
Numerical investigations against experimental results in \cite{Bert2019} show that compared with previous studies, the combined physical effects better describe the propagation speed and the stability transition of the moving droplets.
For higher flow rates, coupled evolution equations of both the film thickness and local flow rate are developed \cite{ruyer2008modelling, ruyer2009film, trifonov1992steady}. These equations incorporate inertia effects and streamwise viscous diffusion based on the integral boundary-layer approach. Recently, Ji et al. \cite{Ji2020Modeling} further extend a weighted-residual integral boundary-layer model to incorporate the film stabilization mechanism to address the effects of the inlet nozzle geometry on the downstream flow dynamics.

While these previous studies primarily focus on the modelling of coating film dynamics, theoretical analysis of these models are still lacking. In this paper, we focus on the model \eqref{r-00} with an additional generalised film stabilization term motivated by \cite{Bert2019,Ji2020Modeling}.
Substituting
$
u = h + r_0
$
into \eqref{r-00} and including the generalised film stabilization term $\tilde{\Pi}(u) = -A/u^m$, 
we obtain a fourth-order nonlinear partial differential equation for $u(t, x)$, namely,
\begin{equation}\label{r-000}
u \, u_{t} +  \left[ \sigma^{-1} Q(u) \left( \frac{u_{xx}}{(1+ u_x^2)^{3/2}} - \frac{1}{u (1+ u_x^2)^{1/2}} + \frac{A}{u^{ m}} \right)_x  + Q(u) \right]_x  = 0 ,
\end{equation}
where the scaling parameter $A \geqslant 0$, the exponent $m > 0$,
and the mobility becomes
$$
Q(u) := \tfrac{1}{4} u^4 \log \bigl( \tfrac{u}{r_0}\bigr) - \tfrac{3}{16}(u^2 - r_0^2)(u^2 -   \tfrac{r_0^2}{3}  ).
$$
A useful observation is that
$$
   Q(u) \sim C\,u^4 \log \bigl( \tfrac{u}{r_0}\bigr)  \text{ as } u \to + \infty \text{ and } u \to r_0.
$$
In this paper, we will show the existence of non-negative weak solutions to \eqref{r-000} with a focus on the travelling wave solutions. Under proper conditions, $u(x,t)$ converges to a travelling wave solution in long time.

The structure of the rest of the paper is as follows. In section \ref{sec:existence}, we present the main theorem on the existence of non-negative weak solutions to the problem, followed by the proof of the existence theorem in section \ref{sec:proof}. Section \ref{sec:tws} focuses on the travelling wave solutions of the problem. Numerical studies based on the analytical results are included in section \ref{sec:numerics}, followed by concluding remarks in section \ref{sec:conclusion}.

\section{Existence of non-negative weak solutions}
\label{sec:existence}
In this section, we will define a generalised non-negative weak solution and formulate an existence of the solution statement for the following problem:
\begin{equation}\label{r-1}
|u| \, u_{t} +  \left( \sigma^{-1} |Q(u)| \left[  ( \Phi'(u_x))_x - \frac{f(u_x)}{u} + \frac{A}{u^{ m}} \right]_x  + |Q(u)| \right)_x  = 0 \text{ in } Q_{T},
\end{equation}
\begin{equation}\label{r-2}
|\Omega| -  \text{periodic boundary conditions},
\end{equation}
\begin{equation}\label{r-3}
u(x,0) = u_0(x),
\end{equation}
where $ \Omega  \subset \mathbb{R}^1$ is bounded domain, $m > 0$, $A \geqslant 0$, $Q_T = \Omega \times (0,T)$, $ T  > 0$, $Q(r_0) = 0$, $f(z) := (1 + z^2)^{-\frac{1}{2}}$, and
$$
\Phi(z) = \frac{1}{f(z)}, \ \Phi'(z) = z\,f(z), \ \Phi''(z) = f^{3}(z) \ \forall\, z \in \mathbb{R}^1.
$$
Assume that
\begin{equation}\label{r-4-000}
 0 \leqslant u_0(x) \in L^2( \Omega ) \cap W_1^1( \Omega ) : \int \limits_{\Omega} {u_0 \Phi(u_{0,x})\,dx} +
A\int \limits_{\Omega} {u^{2-m}_0\,dx} < \infty .
\end{equation}
Integrating (\ref{r-1}) in $\Omega\times (0,t)$, by (\ref{r-2}) we have
\begin{equation}\label{mass-new}
 \int  \limits_{ \Omega } { u^2(x,t) \,  dx }  = \int  \limits_{ \Omega } { u_0^2(x) \,  dx } =: M  \quad \forall \, t \geqslant 0.
\end{equation}
Let us denote by
$$
J:=   ( \Phi'(u_x))_x -  u^{-1} f(u_x) + A\,u^{-m} .
$$
Here, $-J$ represents a dynamic pressure \cite{ji2018instability,Bert2019,ji2020steady} that incorporates both axial and azimuthal surface tension effects and the film stabilization term. 

\begin{definition}\label{def}
A generalised weak solution of the problem (\ref{r-1})--(\ref{r-3}) is a non-negative function
$u(x,t)$ satisfying
\begin{align}
& \label{2:weak1}
u \in C (\overline{Q}_T) \cap L^\infty (0,T; W_1^1(\Omega ))\cap L^\infty (0,T; L^2(\Omega )) ,\\
& \label{2:weak-d} (u^2)_t \in L^2(0,T; (H^1(\Omega))^*),\\
& \label{2:weak2}
|Q(u)|^{\frac{1}{2}} J_x  \in L^2(\mathcal{P}_T),\\
&\label{2:weak2-00}
\chi_{\{|u_x|< \infty \} }  \Phi''(u_{  x})u_{ xx} \in L^2(Q_T), \ u^{-m} \in L^2(Q_T),
\end{align}
where $\mathcal{P}_T = \overline{Q}_T \setminus ( \{u= r_0\} \cup
\{t=0\})$ and $u$ satisfies (\ref{r-1}) in the following weak sense:
\begin{multline} \label{2:integral_form-nn}
\tfrac{1}{2} \int\limits_0^T \langle (u^2(\cdot,t))_t, \phi \rangle_{(H^1)^*,H^1} \; dt -
\sigma^{-1} \iint\limits_{\mathcal{P}_T}
{Q(u) J_{x} \phi_x\,dx dt }  -   
\iint\limits_{Q_{T}} {Q(u) \phi_x  \,dx dt} = 0,
\end{multline}
\begin{equation} \label{int-id-2-00-2}
\iint \limits_{Q_T}{ J  \psi \,dx dt} =
\iint \limits_{Q_T}{ ( \chi_{\{|u_x|< \infty \} }  \Phi''(u_{  x})u_{ xx}  -  \tfrac{ f(u_{x})}{u} + \tfrac{A}{u^m } )  \psi \,dx dt}
\end{equation}
for all $ \phi \in L^2(0,T; H^1(\Omega))$ and $\psi \in L^2(Q_T)$;
\begin{align}
&  \label{2:ID1} u(\cdot,t) \to u(\cdot,0) = u_0
\mbox{  strongly in $L^2(\Omega)$ as $t \to 0$}, \\
& \label{2:BC1}
(\ref{r-2})  \mbox{ holds at all points of
the lateral boundary, where $\{ u  \neq r_0 \}$.}
\end{align}
\end{definition}

Let us denote by $\tilde{G}_0^{(\alpha)}(z)$  the following function
\begin{equation}\label{entropy}
\tilde{G}_0^{(\alpha)}(z):= \int \limits_{s_0}^z { |s| \Bigl( \int \limits_{s_0}^s { \tfrac{|v|^{\alpha}}{|Q(v)|} dv} \Bigr) \,ds} \geqslant 0,
\end{equation}
where $s_0$ is a positive constant. The function  $\tilde{G}_0^{(\alpha)}(z)$ is a generalization of the entropy introduced in \cite{beretta1995nonnegative}.

\begin{theorem}\label{Th-ex}
Assume that the initial function $u_0$ satisfies (\ref{r-4-000}).
Then problem (\ref{r-1})---(\ref{r-3}) has a weak solution $u (x,t)$ defined
in $Q_{T}$ for any $T > 0$, in the sense of Definition~\ref{def}.
Assume that the initial function $u_0 \geqslant r_0 > 0$ also satisfies
$$
\int \limits_{\Omega} {\tilde{G}_0^{(1)}(u_0) \,dx} < \infty .
$$
Then the solution  $u (x,t) $ satisfies $ u \geqslant r_0 >0 $ and
$$
u_t  \in L^2(0,T; (H^1(\Omega))^*),  \ \
\iint \limits_{Q_T}{u \, \Phi''(u_{x})u^2_{xx}\,dx dt} < \infty.
$$
\end{theorem}

\section{Proof of Theorem~\ref{Th-ex}}
\label{sec:proof}

\subsection{Regularised problem}

Note that  (\ref{r-1}) is a parabolic equation which degenerates when $u = 0$, $ u = r_0$ and $|u_x| \to \infty$.
Therefore we have to apply a regularization technique to this equation to overcome the degeneracies.
For given $\varepsilon > 0$ and $\delta > 0$ we have
\begin{multline} \label{rr-1}
g_\varepsilon(u) u_t + \sigma^{-1}(Q_{\varepsilon} (u)[ (\Phi'(u_x))_x +  \delta (g_\varepsilon(u)u_x)_x - \tfrac{g'_\varepsilon(u)}{g_\varepsilon(u)} f(u_x) +\\
A \tfrac{g'_\varepsilon(u)}{g^m_\varepsilon(u)} + \delta  \tfrac{ g'_\varepsilon(u) }{ g_\varepsilon^{\beta}(u) }  ]_x )_x +
(Q_\varepsilon(u))_x =  0   \text{ in } Q_T,
\end{multline}
\begin{equation}\label{rr-2}
|\Omega| -  \text{periodic boundary conditions},
\end{equation}
\begin{equation}\label{rr-3}
u(x,0) = u_{0,\varepsilon \delta}(x),
\end{equation}
where
$$
g_\varepsilon(z):= (z^2 + \varepsilon^2)^{\frac{1}{2}},\  Q_{\varepsilon} (u)  := |Q(u)| + \varepsilon,
$$
$$
u_{0,\varepsilon \delta}(x) \geqslant u_0(x) 
+ \varepsilon^{\theta} +  \delta^\kappa, \ \ u_{0,\varepsilon \delta}(x) \in C^{4+\gamma}(\bar{\Omega}),
$$
$$
u_{0,\varepsilon \delta }(x) \to u_{0,\delta}(x) \text{ strongly in } H^1(\Omega) \text{ as } \varepsilon \to 0,
$$
$$
u_{0,\delta }(x) \to u_{0 }(x) \text{ strongly in } L^2(\Omega) \cap W_1^1(\Omega) \text{ as } \delta \to 0,
$$
where $\gamma \in (0,1)$, $\beta >6$, $\theta \in (0,\frac{1}{3})$, and $\kappa \in (0,\frac{1}{\beta -2})$.

The positivity condition $\varepsilon > 0$ eliminates the degeneracies at $u =0,\, r_0$ in (\ref{r-1}). This
condition  will ``lift'' the initial data and smooth the initial data up to $C^{4+\gamma}(\Omega) $.
The $\delta > 0$ condition  eliminates the degeneracy at $|u_x| \to \infty$ in (\ref{r-1}).
As a result, the regularised equation (\ref{rr-1}) is uniformly parabolic. Moreover, the boundary conditions (\ref{rr-2})
are of Lopatinskii-Shapiro type that implies the existence of a proper continuation of solutions to the whole line (cf. \cite{Sol}, for example).
Using the well-known parabolic Schauder estimates from \cite{Sol}, one can generalise \cite[Theorem 6.3, p. 302]{Ed} and
show that  the regularised problem has a unique classical solution
$u_{\delta \varepsilon}  \in
C_{x,t}^{4+\gamma,1+\gamma/4}( \Omega \times [0, \tau_{\delta \varepsilon}])$ for some time $\tau_{\delta \varepsilon} > 0$.
Here $\tau_{\delta \varepsilon}$ is the local existence time from \cite[Theorem 6.3, p. 302]{Ed}.

\subsection{A priori estimates}
Now, we need to derive a priori estimates for energy-entropy functionals.
Let us denote
$$
J_{\varepsilon \delta}: = (\Phi'(u_x))_x + \delta (g_\varepsilon(u)u_x)_x - \tfrac{g'_\varepsilon(u)}{g_\varepsilon(u)} f(u_x)
+ A \tfrac{g'_\varepsilon(u)}{g^m_\varepsilon(u)} + \delta  \tfrac{ g'_\varepsilon(u) }{ g_\varepsilon^{\beta}(u) } .
$$
Multiplying (\ref{rr-1}) by $- J_{\varepsilon \delta} $ and integrating  over
$\Omega$, we obtain
$$
- \int \limits_{\Omega}{g_\varepsilon(u) u_t J_{\varepsilon \delta} \,dx} + \sigma^{-1} \int \limits_{\Omega}{
 Q_\varepsilon(u)J^2_{\varepsilon \delta, x} \,dx} =   \int \limits_{\Omega}{ ( Q_\varepsilon(u))_{x} J_{\varepsilon \delta} \,dx}.
$$
As
\begin{multline*}
- \int \limits_{\Omega}{g_\varepsilon(u) u_t J_{\varepsilon \delta} \,dx}  =
 \int \limits_{\Omega}{ \Big[ g'_\varepsilon(u) u_x \Phi'(u_x) u_t +
g_\varepsilon(u) \Phi'(u_x) u_{xt} +  g'_\varepsilon(u) f(u_x)u_t \Big] \,dx} + \\
\delta \int \limits_{\Omega}{ \Big[ g'_\varepsilon(u)g_\varepsilon(u) u_x^2 u_t +
  g^2_\varepsilon(u)  u_x u_{xt} \Big] \,dx} +
  \tfrac{A}{m -2} \tfrac{d}{dt} \int \limits_{\Omega}{g^{2-m}_\varepsilon(u) \,dx} +
  \tfrac{\delta}{\beta -2} \tfrac{d}{dt} \int \limits_{\Omega}{g^{2-\beta}_\varepsilon(u) \,dx}  =    \\
\int \limits_{\Omega}{ \Big[  \Phi(u_x) \tfrac{d}{dt} g_\varepsilon(u) + g_\varepsilon(u) \tfrac{d}{dt} \Phi(u_x)\Big]
\,dx} + \\
\tfrac{\delta}{2} \int \limits_{\Omega}{ \Big[ u_x^2  \tfrac{d}{dt} g^2_\varepsilon(u) +
  g^2_\varepsilon(u) \tfrac{d}{dt} u_x^2 \Big] \,dx} + \tfrac{A}{m -2} \tfrac{d}{dt} \int \limits_{\Omega}{g^{2-m}_\varepsilon(u) \,dx}
   +
    \tfrac{\delta}{\beta -2} \tfrac{d}{dt} \int \limits_{\Omega}{g^{2-\beta}_\varepsilon(u) \,dx}   = \\
\tfrac{d}{dt} \int \limits_{\Omega}{ g_\varepsilon(u) \Phi(u_x)\,dx} +
\tfrac{\delta}{2}\tfrac{d}{dt} \int \limits_{\Omega}{ g^2_\varepsilon(u) u_x^2  \,dx} + \tfrac{A}{m -2} \tfrac{d}{dt} \int \limits_{\Omega}{g^{2-m}_\varepsilon(u) \,dx}  +
  \tfrac{\delta}{\beta -2} \tfrac{d}{dt} \int \limits_{\Omega}{g^{2-\beta}_\varepsilon(u) \,dx} ,
\end{multline*}
\begin{multline*}
\int \limits_{\Omega}{ [ Q_\varepsilon(u)]_{x} J_{\varepsilon \delta} \,dx} = - \int \limits_{\Omega}{ Q_\varepsilon (u) J_{\varepsilon \delta,x}  \,dx} \leqslant
\tfrac{\sigma^{-1}}{2} \int \limits_{\Omega}{ Q_\varepsilon(u) J^2_{\varepsilon \delta, x}  \,dx} +
\tfrac{\sigma }{2} \int \limits_{\Omega}{ Q_\varepsilon(u) \,dx}
\end{multline*}
then we get
\begin{equation}\label{as-1}
\tfrac{d}{dt} \mathcal{E}_{\varepsilon\delta}(u)  +
\tfrac{\sigma^{-1}}{2} \int \limits_{\Omega}{
 Q_\varepsilon(u)J^2_{\varepsilon \delta,x} \,dx} \leqslant
\tfrac{\sigma }{2} \int \limits_{\Omega}{ Q_\varepsilon(u) \,dx},
\end{equation}
where
$$
\mathcal{E}_{\varepsilon\delta}(u) :=  \int \limits_{\Omega}{ \Big[ g_\varepsilon(u) \Phi(u_x) + \tfrac{\delta}{2}g^2_\varepsilon(u) u_x^2 + \tfrac{A}{m -2}g^{2-m}_\varepsilon(u) +  \tfrac{\delta}{\beta -2}g^{2-\beta}_\varepsilon(u)} \Big] \,dx.
$$
Integrating (\ref{rr-1}) and taking into account periodic boundary conditions, we get
\begin{equation}\label{mass-new-00}
 \int \limits_{\Omega}{   \tilde{g}_\varepsilon(u)\,dx} =
  \int \limits_{\Omega}{   \tilde{g}_\varepsilon(u_{0,\varepsilon\delta})\,dx} =: M_{\varepsilon\delta},
\text{ where } \tilde{g}'_\varepsilon(z) = g_\varepsilon(z).
\end{equation}
By (\ref{mass-new-00}) we deduce that
$$
\Big| \tilde{g}_\varepsilon(u) - \tfrac{M_{\varepsilon\delta}}{|\Omega|} \Big| =   \Bigl | \int \limits_{x_0}^{x}{ g_\varepsilon(u) u_x \,dx}\Bigr | \leqslant
  \int \limits_{\Omega}{ g_\varepsilon(u) \Phi(u_x)\,dx},
$$
whence
\begin{equation}\label{as-1-00}
\tfrac{1}{2} u^2 \leqslant \tfrac{M_{\varepsilon\delta}}{|\Omega|} +  \int \limits_{\Omega}{ g_\varepsilon(u) \Phi(u_x)\,dx} \Rightarrow
|u| \leqslant \sqrt{2}\Bigl( \tfrac{M_{\varepsilon\delta}}{|\Omega|}  + \int \limits_{\Omega}{ g_\varepsilon(u) \Phi(u_x)\,dx} \Bigr)^{\frac{1}{2}}.
\end{equation}
Taking into account
$$
| Q  (z) | \leqslant C_0  ( 1 + z^4 \log \bigl( \tfrac{z}{r_0}\bigr) ) ,
$$
(\ref{mass-new-00}) and (\ref{as-1-00}), from (\ref{as-1}) we find that
\begin{multline}\label{as-1-01}
\tfrac{d}{dt} \mathcal{E}_{\varepsilon\delta}(u)  +
\tfrac{\sigma^{-1}}{2} \int \limits_{\Omega}{
 Q_\varepsilon(u)J^2_{\varepsilon \delta,x} \,dx} 
 \leqslant \\
\tfrac{\sigma C_0}{2}  \int \limits_{\Omega}{ (1 + u^4 \log \bigl( \tfrac{u}{r_0}\bigr))  \,dx}
\leqslant \tfrac{C_0 \sigma }{2}\bigl[ |\Omega| +   M_{\varepsilon\delta} \sup (u^2 \log \bigl( \tfrac{u}{r_0}\bigr)) \bigr ]
\leqslant \\
\tfrac{C_0 \sigma }{2} |\Omega| +
\tfrac{C_0 \sigma }{2} M_{\varepsilon\delta} \Bigl(\tfrac{M_{\varepsilon\delta}}{|\Omega|} + \int \limits_{\Omega}{ g_\varepsilon(u) \Phi(u_x) \,dx}  \Bigr)  \log \Bigl( \tfrac{2}{r_0^2} \bigl( \tfrac{M_{\varepsilon\delta}}{|\Omega|} + \int \limits_{\Omega}{ g_\varepsilon(u) \Phi(u_x) \,dx} \bigr) \Bigr) 
\leqslant \\
C_{1,\varepsilon\delta} \mathcal{E}_{\varepsilon\delta}(u)  \log (\mathcal{E}_{\varepsilon\delta}(u) ).
\end{multline}
By (\ref{as-1-01}) we obtain
\begin{equation} \label{as-2}
\mathcal{E}_{\varepsilon\delta}(u) +
\tfrac{\sigma^{-1}}{2} \iint \limits_{Q_t}{
 Q_\varepsilon(u)J^2_{\varepsilon \delta,x} \,dx dt}  \leqslant K_{\varepsilon\delta}(t) :=
  \mathcal{E}_{\varepsilon\delta}(u_{0,\varepsilon\delta}) e^{ e^{C_{1,\varepsilon\delta}\,t}}
\end{equation}
for all $t \geqslant 0$.
Let us denote by
$$
G_{\varepsilon}^{(\alpha)}(z)  :
( G^{(\alpha)}_{\varepsilon }(z))'' = \tfrac{g^{\alpha}_\varepsilon(z) }{Q_{\varepsilon } (z)} \ \forall\, z \in \mathbb{R}^1.
$$
Multiplying (\ref{rr-1}) by $ (G^{(\alpha)}_{\varepsilon }(u))'$ and integrating  over $\Omega$, we obtain
\begin{multline*}
\tfrac{d}{dt} \int \limits_{\Omega}{\tilde{G}^{(\alpha)}_{\varepsilon} (u)   \,dx}
- \\
\sigma^{-1} \int \limits_{\Omega}{ g^{\alpha}_\varepsilon(u) u_x \Bigl[ (\Phi'(u_x))_x +  \delta (g_\varepsilon(u)u_x)_x -
\tfrac{g'_\varepsilon(u)}{g_\varepsilon(u)} f(u_x) + A \tfrac{g'_\varepsilon(u)}{g^m_\varepsilon(u)} + \delta  \tfrac{ g'_\varepsilon(u) }{ g_\varepsilon^{\beta}(u) } }\Bigr]_x  \,dx
- \\
\int \limits_{\Omega}{ g^{\alpha}_\varepsilon(u) u_x \,dx} =0 ,
\end{multline*}
where
$
(\tilde{G}^{(\alpha)}_{\varepsilon }(z) )' =  g_\varepsilon(z) \, (G^{(\alpha)}_{\varepsilon } (z))' .
$
By periodic boundary conditions we have
\begin{multline*}
 \tfrac{d}{dt} \int \limits_{\Omega}{\tilde{G}^{(\alpha)}_{\varepsilon  } (u)   \,dx} +
 \sigma^{-1} \int \limits_{\Omega}{g^{\alpha}_\varepsilon(u)\Phi''(u_{x})u^2_{xx}\,dx}  + \\
 \delta  \sigma^{-1}  \int \limits_{\Omega}{g^{\alpha -1}_\varepsilon(u) ( \tilde{g}_\varepsilon(u))^2_{xx}\,dx} =
 - \sigma^{-1} \alpha \int \limits_{\Omega}{g^{\alpha-1}_\varepsilon(u) g'_\varepsilon(u) u_x^2 \Phi''(u_{x})u_{xx}\,dx} + \\
 \sigma^{-1} \int \limits_{\Omega}{g^{\alpha-1}_\varepsilon(u) g'_\varepsilon(u) f(u_x) u_{xx}    \,dx} +
 \sigma^{-1} \alpha \int \limits_{\Omega}{ g^{\alpha-2}_\varepsilon(u) g'^2_\varepsilon(u) u_x^2 f(u_x) \,dx} - \\
\delta  \sigma^{-1} (\alpha -1)  \int \limits_{\Omega}{ g^{\alpha-2}_\varepsilon(u) g'_\varepsilon(u) u_x
( \tilde{g}_\varepsilon(u))_{x} ( \tilde{g}_\varepsilon(u))_{xx} \,dx}
+ \\
 A \sigma^{-1}\int \limits_{\Omega}{g^{\alpha}_\varepsilon(u)\bigl( \tfrac{g'_\varepsilon(u)}{g^m_\varepsilon(u)} \bigr)' u^2_{x}\,dx}
+   \delta \sigma^{-1}\int \limits_{\Omega}{g^{\alpha}_\varepsilon(u)\bigl( \tfrac{g'_\varepsilon(u)}{g^\beta_\varepsilon(u)} \bigr)' u^2_{x}\,dx}    \leqslant  \\
\tfrac{\sigma^{-1}}{2} \int \limits_{\Omega}{g^{\alpha}_\varepsilon(u)\Phi''(u_{x})u^2_{xx}\,dx} +
\sigma^{-1} C_{\alpha} \int \limits_{\Omega}{g^{\alpha-2}_\varepsilon(u) g'^2_\varepsilon(u) \Phi(u_{x})\,dx} +
\end{multline*}

\begin{multline*}
 \tfrac{ \delta \sigma^{-1} (\alpha -1) }{3}  \int \limits_{\Omega}{  \tfrac{ (g_\varepsilon^{\alpha-3}(u) g'_\varepsilon(u) )'  }{g_\varepsilon(u) } (\tilde g_\varepsilon(u) )_x^4 \,dx} + \\
 A \sigma^{-1}\int \limits_{\Omega}{g^{\alpha - m - 3}_\varepsilon(u) [ g^2_\varepsilon(u) - (m+1)u^2 ]  u^2_{x}\,dx}
+  \delta \sigma^{-1}\int \limits_{\Omega}{g^{\alpha - \beta - 3}_\varepsilon(u) [ g^2_\varepsilon(u) - (\beta+1)u^2 ]  u^2_{x}\,dx}
 \leqslant \\
 \tfrac{\sigma^{-1}}{2} \int \limits_{\Omega}{g^{\alpha}_\varepsilon(u)\Phi''(u_{x})u^2_{xx}\,dx} +
  \sigma^{-1} C_{\alpha} \int \limits_{\Omega}{g^{\alpha-2}_\varepsilon(u)  \Phi(u_{x})\,dx} +\\
 \tfrac{ \delta \sigma^{-1}  |\alpha -1| (1+|\alpha-4| ) }{3} \int \limits_{\Omega}{   g_\varepsilon^{\alpha-5}(u)   (\tilde g_\varepsilon(u) )_x^4 \,dx}  +\\
 A \sigma^{-1}\int \limits_{\Omega}{g^{\alpha - m - 3}_\varepsilon(u) [ - m \,g^2_\varepsilon(u) + (m+1)\varepsilon^2 ]  u^2_{x}\,dx}
 +   \delta \sigma^{-1}\int \limits_{\Omega}{g^{\alpha - \beta - 3}_\varepsilon(u) [ - \beta \,g^2_\varepsilon(u) + (\beta+1)\varepsilon^2 ]  u^2_{x}\,dx} ,
\end{multline*}
where $C_{\alpha} =  \alpha^2 + | \alpha | +1 $, whence
\begin{multline} \label{as-2-0-al}
\tfrac{d}{dt} \int \limits_{\Omega}{\tilde{G}^{(\alpha)}_{\varepsilon  } (u)   \,dx} +
m\, A \sigma^{-1}\int \limits_{\Omega}{g^{\alpha - m - 1}_\varepsilon(u) u^2_{x}\,dx}  +
 \delta \, \beta  \sigma^{-1}\int \limits_{\Omega}{g^{\alpha - \beta - 1}_\varepsilon(u) u^2_{x}\,dx} +
\\
 \tfrac{\sigma^{-1}}{2} \int \limits_{\Omega}{g^{\alpha}_\varepsilon(u)\Phi''(u_{x})u^2_{xx}\,dx} +
  \delta  \sigma^{-1}   \int \limits_{\Omega}{g^{\alpha -1}_\varepsilon(u) ( \tilde{g}_\varepsilon(u))^2_{xx}\,dx} \leqslant \\
\sigma^{-1} C_{\alpha} \int \limits_{\Omega}{g^{\alpha-2}_\varepsilon(u)  \Phi(u_{x})\,dx} +
\tfrac{ \delta \sigma^{-1}  |\alpha -1| (1+|\alpha-4| ) }{3} \int \limits_{\Omega}{   g_\varepsilon^{\alpha-5}(u)   (\tilde g_\varepsilon(u) )_x^4 \,dx} +\\
\varepsilon^2 (m+1) A \sigma^{-1}\int \limits_{\Omega}{g^{\alpha - m - 3}_\varepsilon(u)  u^2_{x}\,dx}
+  \varepsilon^2 \delta (\beta+1) \sigma^{-1}\int \limits_{\Omega}{g^{\alpha - \beta - 3}_\varepsilon(u)  u^2_{x}\,dx}.
\end{multline}
From (\ref{as-2-0-al}) we have
\begin{multline} 
\tfrac{d}{dt} \int \limits_{\Omega}{\tilde{G}^{(\alpha)}_{\varepsilon  } (u)   \,dx} +
 \tfrac{\sigma^{-1}}{2} \int \limits_{\Omega}{g^{\alpha}_\varepsilon(u)\Phi''(u_{x})u^2_{xx}\,dx} +\\
 m\, A \sigma^{-1}\int \limits_{\Omega}{g^{\alpha - m - 3}_\varepsilon(u) ( \tilde{g}_\varepsilon(u))^2_{x } \,dx} +
 \delta \, \beta  \sigma^{-1}\int \limits_{\Omega}{g^{\alpha - \beta - 3}_\varepsilon(u) ( \tilde{g}_\varepsilon(u))^2_{x } \,dx}  +
\\
 \delta  \sigma^{-1}  \int \limits_{\Omega}{g^{\alpha -1}_\varepsilon(u) ( \tilde{g}_\varepsilon(u))^2_{xx}\,dx} \leqslant
\sigma^{-1} C_{\alpha} \mathop {\sup} \limits_{\Omega} ( g^{-2}_\varepsilon(u) ) \int \limits_{\Omega}{g^{\alpha}_\varepsilon(u)  \Phi(u_{x})\,dx} + \\
\tfrac{ \delta \sigma^{-1}  |\alpha -1| (1+|\alpha-4| ) }{3} \int \limits_{\Omega}{   g_\varepsilon^{\alpha-5}(u)   (\tilde g_\varepsilon(u) )_x^4 \,dx} +\\
\varepsilon^2 (m+1) A \sigma^{-1}  \mathop {\sup} \limits_{\Omega} ( g^{-2}_\varepsilon(u) ) \int \limits_{\Omega}{g^{\alpha - m - 3}_\varepsilon(u) ( \tilde{g}_\varepsilon(u))^2_{x }\,dx} +
\nonumber
\end{multline}
\begin{equation}
\varepsilon^2 \delta (\beta+1)   \sigma^{-1}  \mathop {\sup} \limits_{\Omega} ( g^{-2}_\varepsilon(u) ) \int \limits_{\Omega}{g^{\alpha - \beta - 3}_\varepsilon(u) ( \tilde{g}_\varepsilon(u))^2_{x }\,dx} .
\label{as-2-0-a2}
\end{equation}

Next, we will use the following lemma:
\begin{lemma}\cite[Lemma 3.1, p.806]{CD}
\label{L1} 
Let $v \in H^1(\Omega)$ be a nonnegative function. Then for any
$p > 2$, there exists a constant $\tilde{C} $ depending on $p$ and $\Omega$ such that
\begin{equation}\label{min-n}
\sup v^{-1} \leqslant \tilde{C} \Bigl[ \Bigl( \int \limits_{\Omega}{ v^{-p} \,dx}  \Bigr)^{\frac{1}{p}} +
\Bigl( \int \limits_{\Omega}{ v^{-p} \,dx}  \Bigr)^{\frac{1}{p-2}}  \Bigl( \int \limits_{\Omega}{ v_x^2 \,dx}  \Bigr)^{\frac{1}{p}} \Bigr] .
\end{equation}
\end{lemma}
Using (\ref{min-n}) with $v = g^{2}_\varepsilon(u)$ and $p = \frac{\beta-2}{2} > 2$, i.\,e. $\beta > 6$, we have
\begin{equation}\label{sup-nn}
\mathop {\sup} \limits_{\Omega} ( g^{-2}_\varepsilon(u) ) \leqslant \tilde{C}
 \Bigl[ \Bigl( \int \limits_{\Omega}{ g^{2-\beta}_\varepsilon(u) \,dx}  \Bigr)^{\frac{2}{\beta-2}} +
\Bigl( \int \limits_{\Omega}{  g^{2-\beta}_\varepsilon(u) \,dx}  \Bigr)^{\frac{2}{\beta-6}}  \Bigl( \int \limits_{\Omega}{g^{2}_\varepsilon(u)  u_x^2 \,dx}  \Bigr)^{\frac{2}{\beta-2}} \Bigr],
\end{equation}
whence, due to (\ref{as-2}), we find that
\begin{equation}\label{sup-nn-2}
\mathop {\sup} \limits_{\Omega} ( g^{-2}_\varepsilon(u) ) \leqslant \tilde{C}_{1,\varepsilon\delta}(t) :=
\tilde{C} \bigl[ \bigl(\tfrac{\beta-2}{\delta} K_{\varepsilon\delta}(t) \bigr)^{\frac{2}{\beta-2}} + 2^{\frac{2}{\beta-2}}
(\beta -2)^{\frac{2}{\beta-6}}
\bigl(\tfrac{1}{\delta} K_{\varepsilon\delta}(t) \bigr)^{\frac{4(\beta -4)}{(\beta -2)(\beta-6)}}   \bigr].
\end{equation}
Choosing $\alpha = 1$ and $\varepsilon $ small enough  in (\ref{as-2-0-a2}), taking into
account (\ref{as-2}) and (\ref{sup-nn-2}),  we deduce that
\begin{multline} \label{as-2-0-a2-00}
\tfrac{d}{dt} \int \limits_{\Omega}{\tilde{G}^{(1)}_{\varepsilon  } (u)   \,dx} +
 \tfrac{\sigma^{-1}}{2} \int \limits_{\Omega}{g_\varepsilon(u)\Phi''(u_{x})u^2_{xx}\,dx}  +
 \delta  \sigma^{-1}  \int \limits_{\Omega}{ ( \tilde{g}_\varepsilon(u))^2_{xx}\,dx} \leqslant\\
3 \sigma^{-1} \tilde{C}_{1,\varepsilon\delta}(t) \, \int \limits_{\Omega}{g_\varepsilon(u)  \Phi(u_{x})\,dx}.
\end{multline}
Integrating (\ref{as-2-0-a2-00}) in time, taking into account (\ref{as-2}), we arrive at
\begin{multline} \label{as-3-000}
 \int \limits_{\Omega}{\tilde{G}^{(1)}_{\varepsilon  } (u)   \,dx} +  \tfrac{\sigma^{-1}}{2} \iint \limits_{Q_t}{g_\varepsilon(u)\Phi''(u_{x})u^2_{xx}\,dx dt} +   \\
 \delta  \sigma^{-1} \iint \limits_{Q_t}{ ( \tilde{g}_\varepsilon(u))^2_{xx}\,dx} \leqslant C_{2,\varepsilon\delta}(t)
\end{multline}
for all $t \geqslant 0$, where
$$
C_{2,\varepsilon\delta}(t) :=
 \int \limits_{\Omega}{\tilde{G}^{(1)}_{\varepsilon } (u_{0,\varepsilon \delta})   \,dx} +
3 \sigma^{-1}   \int \limits_0^t{ \tilde{C}_{1,\varepsilon\delta}(s) K_{\varepsilon\delta}(s)\,ds}.
$$
Let $\tilde{\mathcal{G}}_{\varepsilon  } (z) := \tilde{G}^{(1)}_{\varepsilon  } (z)$ and $\mathcal{G}_{\varepsilon  } (z) := G^{(1)}_{\varepsilon  } (z)$.
Note that
\begin{multline*}
|\mathcal{G}''_{\varepsilon}(z) - \mathcal{G}''_{0}(z)| =
 \Bigl | \tfrac{ g_\varepsilon(z)}{Q_{\varepsilon } (z)} - \tfrac{ |z| }{|Q(z)|}
\Bigr | =
\Bigl | \tfrac{ ( g_\varepsilon(z)- |z|) |Q(z)| - \varepsilon |z| }{|Q(z)| Q_{\varepsilon } (z)} \Bigr |
\leqslant    \\
\tfrac{  g_\varepsilon(z) - |z| }{ Q_\varepsilon(z) } + \tfrac{B\, \varepsilon}{|Q(z)| Q_{\varepsilon } (z)}
\leqslant \tfrac{\varepsilon}{ Q_\varepsilon(z) } + \tfrac{B\, \varepsilon}{|Q(z)| Q_{\varepsilon } (z)} \leqslant
C \tfrac{\varepsilon^{\frac{1}{2}}}{|Q(z)|^{\frac{3}{2}}}
\end{multline*}
provided $|z| \leqslant B $, where $C = \frac{1}{2} (B +C_0(1+B^4 \log(\frac{B}{r_0})))$. So,
\begin{multline*}
|\tilde{\mathcal{G}_{\varepsilon}}(z) - \tilde{\mathcal{G}}_{0}(z)| =  \Bigl | \int \limits_B^z { g_\varepsilon(s) ( \mathcal{G}'_{\varepsilon}(s) - \mathcal{G}'_{0}(s) )\,ds} +
\int \limits_B^z { \mathcal{G}'_{0}(s) (g_\varepsilon(s) -|s|)\,ds} \Bigr |
\leqslant  \\
 C \, \varepsilon^{\frac{1}{2}} \Bigl |  \int \limits_B^z { \int \limits_B^v { \tfrac{ds dv}{|Q(s)|^{\frac{3}{2}}} }} \Bigr | +
 \varepsilon \,  \mathcal{G}_0(z) ,
\end{multline*}
whence, due to $|Q(z)| \sim | \log ( \frac{z}{r_0} )| $ as $z \to r_0$, we find that
$$
|\tilde{\mathcal{G}}_{\varepsilon}(u_{0,\varepsilon \delta}) - \tilde{\mathcal{G}}_{0}(u_{0,\varepsilon \delta})| \leqslant C\, \varepsilon^{\frac{1}{2}}(
\log( 1 + \varepsilon^{\theta} ) )^{-\frac{3}{2}}
\to 0 \text{ as } \varepsilon \to 0
$$
provided $\theta \in (0,\frac{1}{3})$. As a result, we deduce that
\begin{equation}\label{rr-00}
\int \limits_{\Omega} {  \tilde{\mathcal{G}}_{\varepsilon}(u_{0,\varepsilon \delta})\,dx} \to
\int \limits_{\Omega} {  \tilde{\mathcal{G}}_{0}(u_{0,\delta})\,dx} \text{ as } \varepsilon \to 0.
\end{equation}
Therefore, due to (\ref{rr-00}) we have
$$
C_{2,\varepsilon\delta}(t) \mathop \to \limits_{\varepsilon   \to 0 }
C_{2,\delta}(t) .
$$
Note that, after taking limit $\varepsilon \to 0$, we obtain the limit solution
$u_\delta \geqslant r_0 >0$ (the proof is similar to \cite[Theorem~4.1]{B8}), and for this reason, instead of (\ref{sup-nn-2}), we
use $\mathop {\sup} \limits_{\Omega} ( u_{\delta}^{-2}  ) \leqslant r_0^{-2}$ for
the limit process on $\delta \to 0$.

\subsection{Construction of a weak solution}

We will construct a weak non-negative solution using Arzela-Ascoli theorem. By (\ref{as-2}) 
we deduce a uniform boundedness of the following
sequences
\begin{equation}\label{bb-1}
\{  u_{\varepsilon \delta }  \}_{\varepsilon >0, \delta > 0}  \text{ in }
L^{\infty}(0,T; W_1^1(\Omega)),
\end{equation}
\begin{equation}\label{bb-2}
\{ g_\varepsilon(u_{\varepsilon \delta } ) \Phi(u_{\varepsilon \delta,x }) \}_{\varepsilon >0, \delta > 0}  \text{ in } L^{\infty}(0,T; L^1(\Omega)),
\end{equation}
\begin{equation}\label{bb-3}
\{ \delta^{\frac{1}{2}} \tilde{g}_\varepsilon(u_{\varepsilon \delta })  \}_{ \varepsilon >0, \delta > 0}  \text{ in } L^{\infty}(0,T; H^1 (\Omega)),
\end{equation}
\begin{equation}\label{bb-4}
\{ Q^{\frac{1}{2}}_\varepsilon(u_{\varepsilon \delta })J_{\varepsilon \delta,x} \}_{\varepsilon >0, \delta > 0}  \text{ in } L^{2}(Q_T),
\end{equation}
\begin{equation}\label{bb-6}
\{ (\tilde{g}_\varepsilon(u_{\varepsilon \delta }))_t \}_{\varepsilon >0,  \delta > 0}  \text{ in } L^{2}(0,T; (H^1(\Omega))^*).
\end{equation}
\begin{equation}\label{bb-7}
\{  g^{2-m}_\varepsilon(u_{\varepsilon \delta } ) \}_{\varepsilon >0,  \delta > 0}  \text{ in } L^{\infty}(0,T; L^1(\Omega)).
\end{equation}
\begin{equation}\label{bb-7-n}
\{ \delta g^{2-\beta}_\varepsilon(u_{\varepsilon \delta } ) \}_{\varepsilon >0,  \delta > 0}  \text{ in } L^{\infty}(0,T; L^1(\Omega)).
\end{equation}
By (\ref{as-3-000}) and (\ref{bb-3}) we arrive at
\begin{equation}\label{bb-8}
\int \limits_{\Omega}{\tilde{G}^{(1)}_{\varepsilon} (u_{\varepsilon \delta })   \,dx}  \leqslant C ,
\end{equation}
\begin{equation}\label{bb-9}
\{ g^{\frac{1}{2}}_\varepsilon(u_{\varepsilon \delta })  \Phi''^{\frac{1}{2}}(u_{\varepsilon \delta,x })  u_{\varepsilon \delta,xx }  \}_{ \varepsilon >0, \delta > 0}  \text{ in } L^{2}(Q_T),
\end{equation}
\begin{equation}\label{bb-10}
\{  \Phi''^{\frac{1}{2}}(u_{\varepsilon \delta,x })  u_{\varepsilon \delta,xx }    \}_{ \delta > 0}  \text{ in } L^{2}(Q_T),
\end{equation}
\begin{equation}\label{bb-11}
\{ \delta^{\frac{1}{2}} \tilde{g}_\varepsilon(u_{\varepsilon \delta })  \}_{\varepsilon >0, \delta > 0}  \text{ in } L^{2}(0,T; H^2 (\Omega)).
\end{equation}
By (\ref{bb-3}) and (\ref{bb-6}) we have (see, e.\,g. \cite{B8})
\begin{equation}\label{bb-12}
\{ \tilde{g}_\varepsilon(u_{\varepsilon \delta }) \}_{\varepsilon >0}  \text{ is u.b. in }  C^{\frac{1}{2}, \frac{1}{8}}_{x,t}(\bar{Q}_T).
\end{equation}
From (\ref{bb-1}) and (\ref{bb-12}) we obtain that
\begin{equation}\label{con-00}
\tilde{g}_\varepsilon(u_{\varepsilon \delta }) \mathop {\to} \limits_{ \varepsilon \to 0}
\tilde{g}_0(u_{\delta }) := \tfrac{1}{2} u^2_{\delta}  \text{ uniformly  in }  Q_T.
\end{equation}
By (\ref{con-00}) and (\ref{bb-11}) we find that
\begin{equation}\label{con-01}
\tilde{g}_\varepsilon(u_{\varepsilon \delta }) \mathop {\to} \limits_{ \varepsilon \to 0}
\tilde{g}_0(u_{\delta })   \text{ strongly in  }   L^{2}(0,T; H^1(\Omega)),
\end{equation}
\begin{equation}\label{con-02}
\tilde{g}_\varepsilon(u_{\varepsilon \delta }) \mathop {\to} \limits_{ \varepsilon \to 0}
\tilde{g}_0(u_{\delta })   \text{ weakly in  }   L^{2}(0,T; H^2(\Omega)).
\end{equation}
By (\ref{con-00}) and (\ref{bb-8}) we deduce that
\begin{equation}\label{ub-1}
  u_\delta (x,t) \geqslant r_0 > 0 \text{ in }  Q_T.
\end{equation}
By (\ref{con-00}) and (\ref{bb-6}) we get
\begin{equation}\label{con-3-0}
 (\tilde{g}_\varepsilon(u_{\varepsilon \delta }))_t  \mathop {\to} \limits_{\varepsilon \to 0}  (\tilde{g}_0(u_{\delta }))_t \text{ weakly in }  L^{2}(0,T; (H^1(\Omega))^*).
\end{equation}
By (\ref{bb-4}) and (\ref{con-00}) we have
\begin{equation}\label{econ-2-0}
   Q_\varepsilon(u_{\varepsilon\delta} ) J_{\varepsilon\delta,x} \mathop {\to} \limits_{\varepsilon \to 0} \chi_{\{ u_\delta > r_0\}}   Q (u_\delta)  J_{\delta,x} \text{ strongly in }  L^{2}(Q_T).
\end{equation}
By regularity theory of uniformly parabolic equations and by the
uniformly H\"{o}lder continuity of the $u_{\varepsilon\delta}$, we deduce that
\begin{multline}\label{hol}
u_{\varepsilon\delta,t}, \  u_{\varepsilon\delta,x}, \  u_{\varepsilon\delta,xx}, \ u_{\varepsilon\delta,xxx},
u_{\varepsilon\delta,xxxx} \text{ are uniformly convergent } \\
\text{ in any compact subset of } \{ u_\delta > r_0\}.
\end{multline}
It follows that
\begin{equation}\label{fl}
J_{\delta} =   \Phi''(u_{\delta, x})u_{\delta , xx}  +  \delta (\tilde{g}_0(u_{\delta}))_{xx} -
  u_{\delta}^{-1}  f(u_{ \delta, x}) + A\, u_{\delta}^{-m} +  \delta u_{\delta}^{-\beta}
\text{ on } \{ u_\delta > r_0\}.
\end{equation}
Next, multiplying (\ref{rr-1}) by $\phi$ and integrating in $Q_T$, we have
\begin{multline}\label{int-id}
\iint \limits_{Q_T}{ (\tilde{g}_\varepsilon(u_{\varepsilon \delta }))_t \phi \,dx dt}  -  \sigma^{-1}  \iint \limits_{Q_T}{
 Q_\varepsilon(u_{\varepsilon \delta} )J_{\varepsilon \delta,x} \phi_x \,dx dt}  -
\iint \limits_{Q_T}{  Q_\varepsilon(u_{\varepsilon \delta})  \phi_x \,dx dt}  = 0
\end{multline}
\begin{multline}\label{int-id-2}
\iint \limits_{Q_T}{ J_{\varepsilon \delta } \psi \,dx dt} =  \\
\iint \limits_{Q_T}{ \left[ \Phi''(u_{\varepsilon \delta, x})u_{\varepsilon \delta , xx}  +  \delta (\tilde{g}_\varepsilon(u_{\varepsilon \delta}))_{xx} - \tfrac{g'_\varepsilon(u_{\varepsilon \delta})}{g_\varepsilon(u_{\varepsilon \delta})} f(u_{\varepsilon \delta, x}) + A \tfrac{g'_\varepsilon(u)}{g^m_\varepsilon(u)} + \delta \tfrac{g'_\varepsilon(u)}{g^\beta_\varepsilon(u)}   \right]  \psi \,dx dt}
\end{multline}
for all $ \phi \in L^2(0,T; H^1(\Omega))$ and $\psi \in L^2(Q_T)$. Due to (\ref{con-00})--(\ref{hol}),
letting $\varepsilon \to 0$ in (\ref{int-id}) and (\ref{int-id-2}), we arrive at
\begin{multline}\label{int-id-00}
\iint \limits_{Q_T}{ (\tilde{g}_0(u_{\delta }))_t \phi \,dx dt}  -  \sigma^{-1}  \iint \limits_{Q_T }{
\chi_{\{ u_\delta > r_0\}} Q(u_{\delta} )J_{ \delta,x} \phi_x \,dx dt}  -
\iint \limits_{Q_T}{  Q (u_{ \delta})  \phi_x \,dx dt}  = 0
\end{multline}
\begin{equation} \label{int-id-2-00}
\iint \limits_{Q_T}{ J_{\delta } \psi \,dx dt} =
\iint \limits_{Q_T}{ \left[ \Phi''(u_{\delta, x})u_{\delta, xx} + \delta (\tilde{g}_0(u_{\delta}))_{xx}   -
\tfrac{f(u_{\delta, x})}{u_{\delta} } + \tfrac{A}{u^m_{\delta} } +  \tfrac{\delta}{u^\beta_{\delta} }  \right]  \psi \,dx dt}
\end{equation}
for all $ \phi \in L^2(0,T; H^1(\Omega))$ and $\psi \in L^2(Q_T)$.

Next, we pass to the limit $\delta \to 0$ in (\ref{int-id-00}). Note that a solution $u_{\delta}(x,t) \geqslant r_0$ in
$Q_T$. By (\ref{bb-1}), (\ref{bb-2}) and (\ref{mass-new-00}) we get
\begin{equation}\label{con-0}
 u_{ \delta}  \mathop {\to} \limits_{  \delta \to 0} u   \text{ strongly in } C(Q_T),
\end{equation}
\begin{equation}\label{con-0-0}
 u_{ \delta,x}  \mathop {\to} \limits_{\delta \to 0} u_{ x} \text{ a.\,e. in } Q_T,
\end{equation}
\begin{equation}\label{con-1}
 u_{  \delta}  \mathop {\to} \limits_{\delta \to 0} u  \text{*-weakly in } L^{\infty}(0,T; W_1^1(\Omega)) \cap
 L^{\infty}(0,T; L^2(\Omega))
\end{equation}
and a.\,e. in $Q_T$. By (\ref{bb-6}) and (\ref{con-0})
\begin{equation}\label{con-3}
 (\tilde{g}_0(u_{\delta }))_t \mathop {\to} \limits_{\delta \to 0}  (\tilde{g}_0(u))_t \text{ weakly in }
 L^{2}(0,T; (H^1(\Omega))^*).
\end{equation}
By (\ref{bb-10}), (\ref{bb-11}) and (\ref{con-0-0}) we get
\begin{equation}\label{con-4}
\Phi''(u_{\delta, x }) u_{\delta, xx } \mathop {\to} \limits_{\delta \to 0}
\chi_{\{|u_x|< \infty \} } \Phi''(u_{x }) u_{xx } \text{ strongly in }
 L^{2}(Q_T),
\end{equation}
\begin{equation}\label{con-5}
 \delta (\tilde{g}_0(u_{\delta}))_{xx}  \mathop {\to} \limits_{\delta \to 0}
0  \text{ strongly in }  L^{2}(Q_T).
\end{equation}
Letting $\delta \to 0$ in (\ref{int-id-00}) and (\ref{int-id-2-00}),
in view of (\ref{con-0})--(\ref{con-5}), we get a solution $u(x,t)$ which satisfies
(\ref{2:integral_form-nn}) and (\ref{int-id-2-00-2}).

\section{Travelling wave solutions}
\label{sec:tws}

We introduce a change of variables to the reference frame of the travelling wave,
$$
u(x,t) = v(\xi,t), \text{ where } \xi = x - V \, t.
$$
Substituting the change of variables to the PDE \eqref{r-000} leads to the following PDE for $v(\xi, t)$
\begin{equation}\label{tt-1}
v v_t +  \sigma^{-1} [ Q(v) \bigl( (\Phi'(v_{\xi}))_{\xi} -  v^{-1} f(v_{\xi}) +   A v^{-m}  \bigr)_{\xi} ]_{\xi}
+ (\mu Q(v) - \tfrac{V}{2} v^2)_{\xi} = 0
\end{equation}
with $L$-periodic boundary conditions on $\xi$, where $\mu =0$ corresponds to the case without gravity and $\mu = 1$
corresponds to the case with gravity. Note that the total mass of the film $M$ satisfies the conservation-of-mass condition,
$$
\int \limits_0^L {   v^2 (\xi,t) d\xi} = \int \limits_0^L { v_0^2 (\xi) d\xi}  =: M, \text{ where } v_0(\xi) := v (\xi,0).
$$
Let us denote
\begin{equation}
\mathcal{E}(v) := \int \limits_0^L { \{ v \Phi(v_{\xi}) + \tfrac{A}{m -2} v^{2-m}  \} d\xi},
\ \mathcal{L}(v) := \mathcal{E}(v) + \tfrac{\lambda}{2} \int \limits_0^L {   v^2 \, d\xi} \ \  \forall \, \lambda \in \mathbb{R}^1,
\label{eqn:energy}
\end{equation}
$$
P(v) := \{ \xi \in (0,L) : v(\xi) > r_0 \}, \ Z(v) := \{ \xi \in (0,L) : v(\xi) = r_0 \},
$$
$$
J(v) :=  (\Phi'(v_{\xi}))_{\xi} -  v^{-1} f(v_{\xi}) +   A v^{-m}  , \  \bar{M} := \tfrac{M}{L},
$$
\begin{multline*}
\mathcal{M} : = \Bigl \{ v \in L^2(0,L) \cap W_1^1(0,L) :  v \geqslant r_0 > 0,\  v(0) = v(L), 
 v'(0) = v'(L), \ \int \limits_0^L {   v^2   d\xi} = M \Bigr \}.
\end{multline*}
Here, $\mathcal{E}(v)$ is an energy functional that combines the surface tension of the liquid and a local free energy that corresponds to the film stabilization mechanism, $\lambda$ is a Lagrange multiplier, and $\mathcal{L}(v)$ incorporates both $\mathcal{E}(v)$ and the total mass of the film $M = \int_0^L{v^2 d\xi}$. The functional $-J(v)$ represents dynamic pressure of the liquid, and $\bar{M}$ is the average film thickness.

In the subsection \ref{sec:critical_points} -- \ref{sec:convergence_without_gravity}, we consider the case of  $\mu =0$ (without gravity), and the case  $\mu =1$ (with gravity) is studied
 in the subsection \ref{sec:convergence_with_gravity}.

\subsection{Critical points of the energy functional}
\label{sec:critical_points}
\begin{lemma}\label{L-ex}
Let $m > 2$ and $\mu =0$. Then, for every $M$, the functional
$\mathcal{E}$ attains its minimum $v_{\min}$ on  $\mathcal{M}$.
Moreover, if $0 \leqslant A < r_0^{m-1}$ then
$$
\mathcal{E}(v) \leqslant  \tfrac{\sqrt{2} M (r_0^{-1} - A\, \bar{M}^{-\frac{m}{2}} )(A + (m-2)r_0^{m-1}) }{(m-2)(r_0^{m-1} -A) }
\text{ as } t \to +\infty.
$$
\end{lemma}

\begin{proof}[Proof of Lemma~\ref{L-ex}]

First of all, we will show that $ \mathcal{E}(v)$ dissipates.
Multiplying (\ref{tt-1}) by $ - J  $ and integrating on $\xi$, we find that
\begin{equation}\label{tt-2}
\tfrac{d}{dt} \mathcal{E}(v) + \sigma^{-1} \int \limits_{0}^{L} { Q(v) J_{\xi}^2(v) d\xi} =
- \tfrac{V}{2} \int \limits_{0}^{L} { ( v^2)_{\xi} J(v) d\xi} .
\end{equation}
As
$$
\int \limits_{0}^{L} { (v^2)_{\xi} J(v) d\xi} = 0
$$
then by (\ref{tt-2})
\begin{equation}\label{tt-3}
\tfrac{d}{dt} \mathcal{E}(v) + \sigma^{-1} \int \limits_{0}^{L} { Q(v) J_{\xi}^2(v) d\xi} =
0 ,
\end{equation}
whence
$$
\tfrac{d}{dt} \mathcal{E}(v) \leqslant 0 \Rightarrow \mathcal{E}(v(t)) \leqslant \mathcal{E}(v_0(\xi)) \ \ \ \forall\, t \geqslant 0.
$$
Moreover, taking into account
$$
\int \limits_{0}^{L} { v \Phi (v_{\xi})  d\xi} \geqslant r_0 L,
\ \
\int \limits_{0}^{L} {  v^{2-m}  d\xi} \geqslant L\,\bar{M}^{- \frac{m-2}{2}},
$$
we deduce that
$$
\mathcal{E}(v) \geqslant  K_0 := ( r_0  + \tfrac{A}{m -2}\bar{M}^{- \frac{m-2}{2}}  ) L.
$$
The functional $\mathcal{E}(v)$ is non-increasing and bounded from below therefore 
it attains its minimum $v_{\min}$ on the set $\mathcal{M}$.

Let us denote by
$$
\bar{J} : = -\tfrac{1}{L}  \int \limits_0^L {  v^{-1} f(v_{\xi}) \, d\xi}  +
\tfrac{A}{L} \int \limits_0^L {   v^{-m}   \, d\xi} \geqslant -r_0^{-1} + A\, \bar{M}^{-\frac{m}{2}} .
$$
Note that
\begin{multline*}
\int \limits_0^L { (v^2 -r_0^2) ( J(v) - \bar{J})  \, d\xi} = \int \limits_0^L { v^2 J(v)  \, d\xi} - M\,\bar{J}  \\
= -  2 \int \limits_0^L { v  \Phi(v_{\xi}) \, d\xi} +
  \int \limits_0^L {  v \, f(v_{\xi}) \, d\xi}  + A \int \limits_0^L {   v^{2-m}   \, d\xi} - M\,\bar{J} \leqslant 
 -   \int \limits_0^L { v  \Phi(v_{\xi}) \, d\xi} + A  \int \limits_0^L {   v^{2-m}   \, d\xi} - M\,\bar{J} ,
\end{multline*}
whence
$$
\int \limits_0^L { v  \Phi(v_{\xi}) \, d\xi} - A \int \limits_0^L {   v^{2-m}   \, d\xi} \leqslant
- \int \limits_0^L { (v^2 -r_0^2) ( J(v) - \bar{J})  \, d\xi}  - M\,\bar{J}.
$$
As
$$
\int \limits_0^L { v  \Phi(v_{\xi}) \, d\xi} - A \int \limits_0^L {   v^{2-m}   \, d\xi} \geqslant
K_1 \mathcal{E}(v),
$$
where
$$
K_1 := \tfrac{(m-2)(r_0^{m-1} -A)}{A + (m-2)r_0^{m-1}} \in (0,1] \text{ provided } 0 \leqslant A < r_0^{m-1},
$$
then
\begin{equation} \label{tt-5}
 K_1  \mathcal{E}(v) \leqslant  - \int \limits_0^L { (v^2 -r_0^2) ( J(v) - \bar{J})  \, d\xi} - M\,\bar{J}.
\end{equation}
On the other hand, we have
\begin{multline*}
- \int \limits_0^L { (v^2 -r_0^2) ( J(v) - \bar{J})  \, d\xi} = -   \int \limits_0^L { (v^2 -r_0^2)
 \Bigl(  \int \limits_{\xi_0}^{\xi}  {J_{s}(v(s))\, ds } \Bigr)  \, d\xi} 
 \leqslant    \\
\Bigl(  \int \limits_{0}^{L}  {Q(v) J^2_{\xi}(v)\, d\xi } \Bigr)^{\frac{1}{2}}
\int \limits_0^L { (v^2 -r_0^2)
 \Bigl(  \int \limits_{\xi_0}^{\xi}  {\tfrac{ ds }{Q(v(s))}} \Bigr)^{\frac{1}{2}} \, d\xi} \leqslant 
 K_2 \Bigl(  \int \limits_{0}^{L}  {Q(v) J^2_{\xi}(v)\, d\xi } \Bigr)^{\frac{1}{2}}.
\end{multline*}
So, from (\ref{tt-5}) we obtain
\begin{equation}\label{tt-5-00}
K_3 \mathcal{E}^2(v) \leqslant \int \limits_{0}^{L} { Q(v) J_{\xi}^2(v) d\xi} + K_4 ,
\end{equation}
where
$$
K_3 = \tfrac{1}{2} ( \tfrac{K_1}{K_2})^2, \ \ K_4 = ( \tfrac{M  }{K_2})^2 (r_0^{-1} - A\, \bar{M}^{-\frac{m}{2}} )^2_+ .
$$
Then by (\ref{tt-3})  we arrive at
\begin{equation}\label{tt-6}
\tfrac{d}{dt} \mathcal{E}(v) + \sigma^{-1} K_3 \mathcal{E}^2(v) \leqslant
 \sigma^{-1} K_4 .
\end{equation}
Compare (\ref{tt-6}) with a solution of the following problem
\begin{equation}\label{tt-7}
y'(t) \leqslant \alpha \bigl(   \beta^2 -  y^2(t) \bigr)
\text{ with } y(0) =y_0 ,
\end{equation}
where $ y(t) :=  \mathcal{E}(v) > 0$,
$$
\alpha = \sigma^{-1} K_3, \ \beta^2 =  \tfrac{K_4}{K_3} = 2 ( \tfrac{M  }{K_1})^2 (r_0^{-1} - A\, \bar{M}^{-\frac{m}{2}} )^2_+ .
$$
Then we have
$$
y (t) \leqslant  \beta \, \frac{1+ \frac{y_0 -\beta}{y_0 + \beta} e^{-2\alpha \beta t}}{ 1 - \frac{y_0 -\beta}{y_0 + \beta} e^{-2\alpha \beta t}} \ \ \  \forall\, t \geqslant 0.
$$
As a result, we obtain that
\begin{equation}\label{rt-2}
 \mathcal{E}(v) \leqslant \beta  \text{ as } t \to +\infty
\end{equation}
provided
$$
0 \leqslant A < r_0^{m-1}  .
$$
\end{proof}

\subsection{Structure of energy functional minimizers}
\label{sec:minimizers}

The Euler-Lagrange equation for $\mathcal{E}(v)$ under the  constraint $ \int \limits_0^L {   v^2   d\xi} = M$ is given by
\begin{equation}\label{rt-3}
(\Phi'(v' ))'  - v^{-1}f(v' ) + A v ^{-m} = \lambda,
\end{equation}
where $\lambda$ is the Lagrange multiplier. We also need to incorporate the
constraint $v \geqslant r_0 > 0$. If $v \in \mathcal{M}$, we decompose $(0,L)$ according to the value of $v$ into two sets, $P(v)$ and $Z(v)$.

We compute the first variation of
$\mathcal{L}(v)$
about $v_{\min}$ and obtain
$$
\tfrac{d}{d\epsilon} \mathcal{L}(v_{\min} + \epsilon) \Bigl |_{\epsilon = 0} =
\int \limits_0^L {  \{ v  \Phi'(v_{\xi}) \phi'  + (\Phi(v_{\xi}) - A v^{1-m}  + \lambda v)\phi   \}   d\xi},
$$
where $\phi$ is a smooth test function supported in $P(v)$. 
Since the first variation of 
$\mathcal{L}(v)$
along
$v_{\min} + \epsilon$ must vanish for every such test function $\phi$, the Euler-Lagrange equation (\ref{rt-3}) holds.

Similarly, when we take $\phi$ to be a smooth test function supported in $Z(v)$, the first variation of $\mathcal{L}(v)$ along $v_{\min} + \epsilon$ must vanish for every such $\phi$, and we obtain 
$$
\lambda = A r_0^{- m} - r_0^{-1} \text{ on } Z(v).
$$
\hspace{-0.1in} 
Therefore, the energy minimizer $v_{\min}$ satisfies  \eqref{rt-3} on $(0,L)$ in the sense of distributions as
$\lambda$ in (\ref{rt-3}) is a piece-wise constant function, namely,
$$
\lambda(\xi) =  \left \{
\begin{gathered}
\hfill \lambda^* \ \ \  \text{ for }  \xi \in P(v),  \\
 A r_0^{- m} - r_0^{-1}  \text{ for }  \xi \in Z(v).
\end{gathered}
 \right.
$$

\begin{lemma}\label{Lem-1}
Let $m > 2$, $\mu =0$, and $0 \leqslant A <  r_0^{m-1}$.
If $v_{\min}$ minimizes $\mathcal{E}$ on $\mathcal{M}$, then it solves (\ref{rt-3}) with $\lambda(\xi)$ on $(0,L)$.
The Lagrange multiplier $\lambda(\xi)$ is negative and $\lambda^*$ satisfies
\begin{equation}
 A r_0^{ -m} - r_0^{-1} \leqslant \lambda^* = - \tfrac{2}{r_0^2}(r_0 - C_0 + \tfrac{A}{m-2}r_0^{2-m}),
\label{eq:A>0lambda}
\end{equation}
where
$$
\tfrac{1}{2}(r_0 + \tfrac{A\,m}{m-2}r_0^{2-m}) \leqslant  C_0 < r_0 + \tfrac{A}{m-2}r_0^{2-m}.
$$
Furthermore, $v_{\min}$  is of class $C^1$, and $v'_{\min} =0$ on $ \partial P(v_{\min})$.
\end{lemma}

\begin{proof}[Proof of Lemma~\ref{Lem-1}]

Next, we consider (\ref{rt-3}) on  $P(v)$. By the substitution
$$
v'(\xi) = z(v) \neq 0,
$$
we have
\begin{multline} \label{t-2}
z f^3(z) z'  - v^{-1}f(z) + A v^{-m} = \lambda \Leftrightarrow 
 (f(z))' + v^{-1}f(z) - A v^{-m} = - \lambda \\\Leftrightarrow 
  (v f(z))' = A v^{1-m} - \lambda v  , \text{ where } v \neq 0.
\end{multline}
On the other hand, if $v' = 0$ then
$$
v =  \bar{M}^{\frac{1}{2}} \text{ and } \lambda =  \bar{M}^{-\frac{m}{2}} (A -  \bar{M}^{\frac{m-1}{2}} ).
$$
The equation (\ref{t-2}) has the following general solution
\begin{equation}\label{t-2-0}
f(z) = \tfrac{A}{2-m} v^{1-m} - \tfrac{\lambda}{2} v  + C_0 v^{-1} ,
\end{equation}
where $ C_0 \in \mathbb{R}^1$.
For the rest of the proof, we will separately consider the case $A = 0$ and the case $A > 0$, $m > 2$.

\underline{\emph{Case $A = 0$}}: In this case, by (\ref{t-2-0}) we get
\begin{equation}
    f(z) =  v^{-1}( C_0 - \tfrac{\lambda}{2} v^2)
    \label{eqn:A=0_f}
\end{equation}
provided
$$
v^2 + \tfrac{2}{\lambda} v - \tfrac{2C_0}{\lambda} \leqslant 0 \Rightarrow
v_1  \leqslant v \leqslant v_2 , \  \  - \tfrac{1}{2} \leqslant \lambda C_0 < 0,
$$
where
\begin{equation}
v_1:= -\tfrac{1}{\lambda}(1 - \sqrt{1 + 2C_0 \lambda} )  , \ v_2 :=  -\tfrac{1}{\lambda}(1 + \sqrt{1 + 2C_0 \lambda} ) .
\label{eqn:A=0v1v2}
\end{equation}
Note that here we only consider the case $\lambda < 0$ and $C_0 > 0$. This is because for $\lambda \geqslant 0$ or $C_0 \leqslant 0$, equation \eqref{eqn:A=0_f} does not have any real-valued non-trivial periodic solution $v \geqslant 0$. In particular, since
$0 \leqslant f(z) = (1+z^2)^{-\frac{1}{2}} \leqslant 1$, the case $\lambda \geqslant 0$ and $C_0 \leqslant 0$ cannot lead to any real-valued solution $v \geqslant 0$ to  \eqref{eqn:A=0_f}.
For the case $\lambda > 0$ and $C_0 > 0$, we have $v \leqslant v_1$ or $v \geqslant v_2$ where $v_1$ and $v_2$ are defined in \eqref{eqn:A=0v1v2} and $v_1 < 0 < v_2$. Smooth periodic solutions that satisfy these criteria do not exist. Moreover, for $C_0 < 0$ and $\lambda < 0$, we have $v_1 < 0 <  v \leqslant v_2$ and non-negative smooth periodic solutions do not exist for this case.

From here we arrive at
$$
z^2(v) = v'^2(\xi) = - \frac{(v^2 - v_1^2)(v^2 - v_2^2)}{(v^2 - \frac{2C_0}{\lambda})^2}.
$$
This equation has a periodic solution with the period
\begin{equation}\label{t-tau}
\tau = 2 \int \limits_{v_1}^{v_2} { \sqrt{ - \frac{(s^2 - \frac{2C_0}{\lambda})^2}{(s^2 - v_1^2)(s^2 - v_2^2)} }\,  ds}.
\end{equation}
In the limit case when the droplet touches the substrate surface, we have the minimum and the maximum of the periodic solution given by
\begin{equation}
    v_1 = r_0, \quad v_2 = \frac{  C_0 r_0 }{r_0 - C_0}, \quad \text{where}~ \frac{r_0}{2} \leqslant  C_0 < r_0,
    \label{eq:A=0v1v2}
\end{equation}
and the corresponding $\lambda^*$ satisfies
\begin{equation}
    -r_0^{-1} \leqslant \lambda^* = - \frac{2(r_0 - C_0)}{r_0^2}.
    \label{eq:lambstar}
\end{equation}
A typical smooth periodic solution $v(\xi)$ defined on $a \leqslant \xi \leqslant a + \tau$ with $v(a)= v(a +\tau) = v_1 = r_0$, $v'(a)= v'(a +\tau) = 0$, and $P(v) =(a,a+\tau) \subset (0, L) $
is given by
\begin{equation}\label{ex-sol}
\begin{gathered}
\int \limits_{v_1}^{v(\xi)} { \sqrt{ - \frac{(s^2 - \frac{2C_0}{\lambda})^2}{(s^2 - v_1^2)(s^2 - v_2^2)} }\,  ds} = \xi
\text{ for } \xi \in [a, a + \tfrac{\tau}{2}],     \\
\int \limits_{v_2}^{v(\xi)} { \sqrt{ - \frac{(s^2 - \frac{2C_0}{\lambda})^2}{(s^2 - v_1^2)(s^2 - v_2^2)} }\,  ds} = a + \tfrac{\tau}{2} - \xi
\text{ for } \xi \in [a + \tfrac{\tau}{2}, a+ \tau].
\end{gathered}
\end{equation}
If we set $C_0 = \tfrac{r_0}{2}$, then from \eqref{eq:lambstar} and \eqref{ex-sol} we obtain $\lambda^* = -r_0^{-1}$ and the trivial solution $v \equiv r_0$.

\underline{\emph{Case $A > 0,\, m > 2$}}: Similar to the case $A = 0$,  we only need to consider the case $\lambda < 0$ and $C_0 > 0$. In this case, by (\ref{t-2-0}) we get
$$
f(z) =  v^{1-m}( - \tfrac{\lambda}{2} v^m  + C_0 v^{m -2} - \tfrac{A}{m-2})
$$
provided
\begin{multline}\label{rr-tt}
0 \leqslant - \tfrac{\lambda}{2} v^m  + C_0 v^{m -2} - \tfrac{A}{m-2} \leqslant v^{m-1} \Rightarrow    \\
 g(v) := v^{m-2}(v-v_1)(v -v_2) \leqslant - \tfrac{2A}{\lambda(m-2)}.
\end{multline}
So, if
$$
\max \{- \tfrac{\lambda(m-2)}{2} g(\tilde{v}_{\min}), 0 \} < A \leqslant A^* := - \tfrac{\lambda(m-2)}{2} g(\tilde{v}_{\max}),
$$
where
$$
 \tilde{v}_{\max}:= - \tfrac{m-1}{m \lambda}(1- \sqrt{1 + \tfrac{2m(m-2) }{(m-1)^2}\lambda C_0 }),\quad
  \tilde{v}_{\min}:= - \tfrac{m-1}{m \lambda}(1+ \sqrt{1 + \tfrac{2m(m-2)}{(m-1)^2}\lambda C_0 }),
$$
and $ - \tfrac{(m-1)^2}{2m(m-2)} < \lambda C_0 < 0$, then there exist  $v_1^*(A) $ and $ v_2^*(A) $ such that
$$
0 < v_1^*(A) < v_1 < v_2 < v_2^* (A) , \  \ v_i^* (A) \to v_i \text{ as } A \to 0,
$$
and (\ref{rr-tt}) is true for $v_1^*(A) \leqslant v \leqslant v_2^* (A)  $.
From here we arrive at
$$
z^2(v) = v'^2(\xi) =
- \tfrac{[v^{m-2}(v -v_1)(v - v_2) + \frac{2A}{\lambda(m-2)} ][v^{m-2} (v + v_1)(v + v_2) + \frac{2A}{\lambda(m-2)}]}
{[v^{m-2}(v^2 - \frac{2C_0}{\lambda}) + \frac{2A}{\lambda(m-2)}]^2}.
$$
This equation has a smooth periodic solution with the period
\begin{equation}\label{t-tau-2}
\tau = 2 \int \limits_{v_1^*(A)}^{v_2^*(A)} { \sqrt{ - \tfrac{[s^{m-2}(s^2 - \frac{2C_0}{\lambda}) + \frac{2A}{\lambda(m-2)}]^2}{[s^{m-2}(s -v_1)(s - v_2) + \frac{2A}{\lambda(m-2)} ][s^{m-2}(s + v_1)(s + v_2) + \frac{2A}{\lambda(m-2)}]}
}\,  ds}.
\end{equation}
Taking the minimum of the solution $v^*_1(A) = r_0$, we have
\begin{equation}
A r_0^{ -m} - r_0^{-1} \leqslant \lambda^* = - \tfrac{2}{r_0^2}(r_0 - C_0 + \tfrac{A}{m-2}r_0^{2-m}),
\end{equation} 
\begin{equation}
    \tfrac{1}{2}(r_0 + \tfrac{A\,m}{m-2}r_0^{2-m}) \leqslant  C_0 < r_0 + \tfrac{A}{m-2}r_0^{2-m}, \
0 \leqslant  A < r_0^{m-1}.
\end{equation}
If we set $v(a)= v(a +\tau) = v_1 = r_0$ and $P(v) =(a,a+\tau) \subset (0, L) $, then
our smooth periodic solution $v(\xi)$ is given by
$$
\int \limits_{v_1^*(A)}^{v(\xi)} { \sqrt{ - \tfrac{[s^{m-2}(s^2 - \frac{2C_0}{\lambda}) + \frac{2A}{\lambda(m-2)}]^2}{[s^{m-2}(s -v_1)(s - v_2) + \frac{2A}{\lambda(m-2)} ][s^{m-2}(s + v_1)(s + v_2) + \frac{2A}{\lambda(m-2)}]}
} \,  ds} = \xi \text{ for } \xi \in [a, a+ \tfrac{\tau}{2}],
$$
$$
\int \limits_{v_2^*(A)}^{v(\xi)} { \sqrt{ - \tfrac{[s^{m-2}(s^2 - \frac{2C_0}{\lambda}) + \frac{2A}{\lambda(m-2)}]^2}{[s^{m-2}(s -v_1)(s - v_2) + \frac{2A}{\lambda(m-2)} ][s^{m-2}(s + v_1)(s + v_2) + \frac{2A}{\lambda(m-2)}]}
} \,  ds} = a+ \tfrac{\tau}{2} - \xi \text{ for } \xi \in [a+ \tfrac{\tau}{2}, a+ \tau],
$$
and $v'(a) = v'(a+\tau) = 0$.
Furthermore, if $C_0 = \frac{1}{2}(r_0 + \frac{A\,m}{m-2}r_0^{2-m})$ then we obtain the trivial solution $v \equiv r_0$ on $(0,L)$.

On the other hand, if $\lambda > 0$   and $C_0 > 0$ then
$$
f(z) =  v^{1-m}( - \tfrac{\lambda}{2} v^m  + C_0 v^{m -2} - \tfrac{A}{m-2})
$$
provided
\begin{multline*}
0 \leqslant - \tfrac{\lambda}{2} v^m  + C_0 v^{m -2} - \tfrac{A}{m-2} \leqslant v^{m-1} \Rightarrow    \\
 g(v) := v^{m-2}(v-v_1)(v -v_2) \geqslant - \tfrac{2A}{\lambda(m-2)}.
\end{multline*}
Since $v_2 < 0 < v_1  $, smooth periodic solutions do not exist.

If $\lambda =0$   and $C_0 > 0$ then
$$
f(z) =  v^{1-m}( C_0 v^{m -2} - \tfrac{A}{m-2})
$$
provided
\begin{multline*}
0 \leqslant C_0 v^{m -2} - \tfrac{A}{m-2} \leqslant v^{m-1} \Rightarrow    
v \geqslant ( \tfrac{A}{C_0(m-2)})^{\frac{1}{m-2}} \text{ and }  v^{m-2}(v- C_0) \geqslant - \tfrac{ A}{ m-2 }.
\end{multline*}
Therefore, we do not have any smooth periodic solutions.
\end{proof}

\subsection{Convergence to an energy minimizer without gravity}
\label{sec:convergence_without_gravity}
Next we study the long-time behavior of solutions to \eqref{tt-1}.
We start by considering the case without the gravitational term ($\mu = 0$).
\begin{theorem}\label{Th-conv}
Let $m > 2$, $\mu =0$, $v(\xi,t)$ be a weak solution to (\ref{tt-1}) with periodic boundary conditions,
and $v_{\min}(\xi)$ be a solution from Lemma~\ref{Lem-1}. Assume that
$$
0 \leqslant A < r_0^{m-1}, \ \mathcal{E}(v_{\min}) \geqslant \mathcal{E}^* := \tfrac{ M (r_0^{-1} - A\, \bar{M}^{-\frac{m}{2}} )(A + (m-2)r_0^{m-1}) }{(m-2)(r_0^{m-1} -A) }.
$$
Then there exist $B_0 > 0$ and $B_1 > 0$  such that
$$
0 \leqslant \mathcal{E}(v) -\mathcal{E}(v_{\min}) \leqslant \tfrac{B_0 }{1 + B_1\, t},
$$
and
\begin{equation}\label{cc-8}
v (.,t) \to v_{\min}(.) \text{ weakly in } W^1_1(0,L)\text{ as } t \to +\infty .
\end{equation}
\end{theorem}

\begin{proof}[Proof of Theorem~\ref{Th-conv}]

Let us denote
$$
\mathcal{E}(v | v_{\min}) := \mathcal{E}(v) -\mathcal{E}(v_{\min}).
$$
Similar to (\ref{tt-5-00}), we deduce that
\begin{equation}\label{tt-5-01}
K_3 \mathcal{E}^2(v | v_{\min}) \leqslant \int \limits_{0}^{L} { Q(v) J_{\xi}^2(v) d\xi} + \tilde{K}_4 ,
\end{equation}
where
$$
K_3 = \tfrac{1}{2} ( \tfrac{K_1}{K_2})^2, \ \ \tilde{K}_4 =
 \tfrac{1}{K_2^2} \bigl( M(r_0^{-1} - A\, \bar{M}^{-\frac{m}{2}} ) - K_1  \mathcal{E}(v_{\min}) \bigr)^2_+ .
$$
Then by (\ref{tt-5-01}), similar to (\ref{tt-6}),  we arrive at
\begin{equation}\label{tt-6-01}
 \tfrac{d}{dt} \mathcal{E}(v | v_{\min}) + \sigma^{-1} K_3 \mathcal{E}^2(v | v_{\min}) \leqslant  \sigma^{-1} \tilde{K}_4 .
\end{equation}
Compare \eqref{tt-6-01} with a solution of the following problem
\begin{equation}\label{tt-7-01}
y'(t) \leqslant    \alpha \bigl( \beta^2- y^2(t) \bigr)
\text{ with } y(0) =y_0 ,
\end{equation}
where $ y(t) :=   \mathcal{E}(v | v_{\min}) \geqslant 0$,
$$
\alpha = \sigma^{-1} K_3 , \ \beta^2 = \tfrac{\tilde{K}_4}{ K_3}.
$$
As $\beta = 0$ provided $ \mathcal{E}(v_{\min}) \geqslant \mathcal{E}^*$
then from (\ref{tt-7-01}) we have
\begin{equation}\label{tt-7-02}
y'(t) \leqslant   - \alpha   y^2(t) \text{ with } y(0) =y_0 .
\end{equation}
Solving (\ref{tt-7-02}), we deduce that
$$
y(t) \leqslant \frac{y_0}{1 + \alpha y_0   \, t}   .
$$
As a result,   we arrive at
\begin{equation}\label{cc-7}
0 \leqslant \mathcal{E}(v | v_{\min}) \leqslant \frac{B_0}{1 + B_1\, t}
\to 0 \text{ as } t \to +\infty
\end{equation}
where
$$
B_0 =  \mathcal{E}(v_0 | v_{\min}) , \quad \ B_1 = \sigma^{-1} \mathcal{E}(v_0 | v_{\min}) K_3 ,
$$
provided
$$
0 \leqslant A < r_0^{m-1}
\text{ and } \mathcal{E}(v_{\min}) \geqslant \mathcal{E}^*.
$$
Since the energy $\mathcal{E}(v)$ is bounded by (\ref{cc-7}) and $v\geqslant r_0 > 0$, it follows that $\| v \|_{W_1^1(0,L)} \leqslant \bar{C}$, where $\bar{C}$ is a constant, and the conclusion
\eqref{cc-8} holds.
\end{proof}

\subsection{Convergence to a travelling wave with gravity}
\label{sec:convergence_with_gravity}
The case with the gravitational term $(\mu = 1)$ is more complicated and will require us to introduce and additional conditions 
to compare to the case without gravity. 
Let us denote by
\begin{equation}
F(\xi,t) :=   - \sigma \int \limits_0^{\xi} { (1 - \tfrac{V}{2} \tfrac{v^2(y)}{Q(v(y))} + \tfrac{\nu}{Q(v(y))})\, dy }  \ \ \forall\,\nu \in \mathbb{R}^1,
\label{eqn:F_def}
\end{equation}
\begin{equation}
\tilde{\mathcal{E}}(v(t)) := \int \limits_0^{L} { \{  v \Phi(v_{\xi}) + \tfrac{A}{m-2}v^{2-m}
+ \tfrac{1}{2} v^2  F (\xi,t) \} \, d\xi },
\label{eqn:modified_E}
\end{equation}
where $\tilde{\mathcal{E}}(v)$ is a modified energy functional that incorporates surface energy, the stabilization mechanism, and a gravitational potential energy represented by $\int_0^L \tfrac{1}{2}v^2 F(\xi,t)~d\xi$.

\begin{lemma}\label{L-ex-gr}
Let $m > 2$ and $\mu =1$. Assume that
\begin{equation}\label{ss-11-00}
F(0,t) = F (L,t)  \Leftrightarrow \int \limits_0^{L} {(\tfrac{V}{2} \tfrac{v^2(y)}{Q(v(y))} - \tfrac{\nu}{Q(v(y))} ) \, d\xi } = L,
\end{equation}
and
\begin{equation}\label{ss-11}
\int \limits_0^{L} {   v^2  F_t(\xi,t) \, d\xi } \leqslant 0 .
\end{equation}
Then, for every $M$, the functional
$\tilde{\mathcal{E}}$ attains its minimum $v_{\min}$ on  $\mathcal{M}$.
\end{lemma}

\begin{proof}[Proof of Lemma~\ref{L-ex-gr}]

Multiplying (\ref{tt-1}) with $\mu =1$ by $- (J - F(\xi,t))$ and integrating on $\xi$, we deduce that
\begin{equation}\label{ss-10}
\tfrac{d}{dt} \tilde{\mathcal{E}}(v(t)) - \tfrac{1}{2} \int \limits_0^{L} {   v^2  F_t(\xi,t) \, d\xi } +
\sigma^{-1}\int \limits_0^{L} { Q(v)( J(v) - F(\xi,t))^2_{\xi} \, d\xi }  =0.
\end{equation}
Then from (\ref{ss-10}), due to (\ref{ss-11}), we get
\begin{equation}\label{ss-12}
\tfrac{d}{dt} \tilde{\mathcal{E}}(v(t)) \leqslant 0.
\end{equation}
Moreover, taking into account
$$
  \int \limits_{0}^{L} { v \Phi (v_{\xi})  d\xi} \geqslant r_0 L,\ \
\int \limits_{0}^{L} {  v^{2-m}  d\xi} \geqslant L\,\bar{M}^{- \frac{m-2}{2}}, \ \
\int \limits_0^{L} { \  v^2  F (\xi,t)  \, d\xi } \geqslant   - \sigma L  M,
$$
we deduce that
$$
\tilde{\mathcal{E}}(v) \geqslant  \tilde{K}_0 := ( r_0  + \tfrac{A}{m -2}\bar{M}^{- \frac{m-2}{2}}    - \tfrac{1}{2} \sigma M ) L.
$$
As the functional $\tilde{\mathcal{E}}(v)$ is non-increasing and bounded from bottom then
it attains its minimum $v_{\min}$ on the set $\mathcal{M}$.
\end{proof}

The Euler-Lagrange equation for $\tilde{\mathcal{E}}(v)$
under the  constraint $ \int \limits_0^L {   v^2   d\xi} = M$
is given by
\begin{equation}\label{rt-3-ss}
(\Phi'(v' ))'  - v^{-1}f(v' ) + A v ^{-m} = F_{\infty}(\xi) + \lambda  ,
\end{equation}
where $\lambda$ is a Lagrange multiplier and
$$
F_{\infty}(\xi) :=  \mathop {\lim} \limits_{t \to +\infty} F(\xi,t).
$$
 We also need to incorporate the
constraint $v \geqslant r_0 > 0$. If $v \in \mathcal{M}$, we decompose $(0,L)$ according to the value of $v$ into the
set $P(v)$ and the  set $Z(v)$. As equation (\ref{rt-3-ss}) is not autonomous then $meas (Z(v)) =0$.

\begin{lemma}\label{Lem-1-stw}
Let $m > 2$, $\mu =1$, and $0 \leqslant A < r_0^{m-1}$. The functional $\tilde{\mathcal{E}}$ has the minimizer $v_{\min}$  on $\mathcal{M}$
such that $v_{\min} \in C^1[0,L]$, $v_{\min}(0) = v_{\min}(L) $ and $v'_{\min}(0) = v'_{\min}(L)=0$, and
solves (\ref{rt-3-ss}) with
$$
 A r_0^{ -m} - r_0^{-1} \leqslant \lambda  = - \tfrac{2}{r_0^2}(r_0 - C_0 + \tfrac{A}{m-2}r_0^{2-m}) < 0,
$$
where
$$
\tfrac{1}{2}(r_0 + \tfrac{A\,m}{m-2}r_0^{2-m}) \leqslant  C_0 < r_0 + \tfrac{A}{m-2}r_0^{2-m},
$$
provided
$$
V = 2 \frac{L \int \limits_0^{L} {\tfrac{v_{\min}^2 d\xi}{Q(v_{\min})}   } - M \int \limits_0^{L} { \tfrac{ d\xi}{Q(v_{\min})} }  }{ (\int \limits_0^{L} { \tfrac{v_{\min}^2 d\xi}{Q(v_{\min})}   })^2 - ( \int \limits_0^{L} { \tfrac{ d\xi}{Q(v_{\min})} } ) ( \int \limits_0^{L} { \tfrac{v_{\min}^4 d\xi}{Q(v_{\min})}  } ) },
$$
$$
\nu =   \frac{L \int \limits_0^{L} {\tfrac{v_{\min}^4 d\xi}{Q(v_{\min})}   } - M \int \limits_0^{L} { \tfrac{ v_{\min}^2 d\xi}{Q(v_{\min})} }  }{ (\int \limits_0^{L} { \tfrac{v_{\min}^2 d\xi}{Q(v_{\min})}   })^2 - ( \int \limits_0^{L} { \tfrac{ d\xi}{Q(v_{\min})} } ) ( \int \limits_0^{L} { \tfrac{v_{\min}^4 d\xi}{Q(v_{\min})}  } ) } .
$$
\end{lemma}

\begin{proof}[Proof of Lemma~\ref{Lem-1-stw}]

Note that equation (\ref{rt-3-ss}) has a periodic solution provided
\begin{equation}\label{ss-3}
F_{\infty}(0) = F_{\infty}(L) =0  \Leftrightarrow \int \limits_0^{L} {(\tfrac{V}{2} \tfrac{v_{\min}^2}{Q(v_{\min})} - \tfrac{\nu}{Q(v_{\min})} ) \, d\xi } = L.
\end{equation}
Multiplying (\ref{rt-3-ss}) by $-v\,v_{\xi} $ $(v_{\xi} \neq 0)$, we obtain that
\begin{equation}\label{ss-4}
\bigl( v   f(v_{\xi}) +   \tfrac{A}{m-2} v^{2-m} \bigr)_{\xi}    = - ( F_{\infty}(\xi) + \lambda)  v v_{\xi}.
\end{equation}
We will look for the first integral to (\ref{ss-4}) in the form:
\begin{equation}\label{ss-5}
 f(v_{\xi}) = a(\xi) v + b(\xi) v^{-1} - \tfrac{A}{m-2} v^{1-m},
\end{equation}
where $a(0)=a(L)$ and $b(0) = b(L)$. Substituting (\ref{ss-5}) into (\ref{ss-4}), we find that
$$
a(\xi) = - \tfrac{1}{2} ( F_{\infty}(\xi) + \lambda)  , \ \ b(\xi) =  C_0 + \tfrac{1}{2} \int \limits_0^{\xi}{ F'_{\infty}(y) v^2(y)\,dy}
 \ \   \forall\, C_0 \in \mathbb{R}^1 .
$$
By $b(0) = b(L)  $   we have
\begin{equation}\label{ss-6}
 \int \limits_0^{L} {(\tfrac{V}{2} \tfrac{v_{\min}^4 }{Q(v_{\min})} - \tfrac{\nu v_{\min}^2}{Q(v_{\min})} ) \, d\xi } = M.
\end{equation}
In particular, from (\ref{ss-3}) and (\ref{ss-6}) it follows that
$$
V = 2 \frac{L \int \limits_0^{L} {\tfrac{v_{\min}^2 d\xi}{Q(v_{\min})}   } - M \int \limits_0^{L} { \tfrac{ d\xi}{Q(v_{\min})} }  }{ (\int \limits_0^{L} { \tfrac{v_{\min}^2 d\xi}{Q(v_{\min})}   })^2 - ( \int \limits_0^{L} { \tfrac{ d\xi}{Q(v_{\min})} } ) ( \int \limits_0^{L} { \tfrac{v_{\min}^4 d\xi}{Q(v_{\min})}  } ) },
$$
$$
\nu =   \frac{L \int \limits_0^{L} {\tfrac{v_{\min}^4 d\xi}{Q(v_{\min})}   } - M \int \limits_0^{L} { \tfrac{ v_{\min}^2 d\xi}{Q(v_{\min})} }  }{ (\int \limits_0^{L} { \tfrac{v_{\min}^2 d\xi}{Q(v_{\min})}   })^2 - ( \int \limits_0^{L} { \tfrac{ d\xi}{Q(v_{\min})} } ) ( \int \limits_0^{L} { \tfrac{v_{\min}^4 d\xi}{Q(v_{\min})}  } ) } .
$$
As a result, by (\ref{ss-5}) we arrive at
\begin{equation}\label{ss-8}
  v^2_{\xi} = \frac{1 - ( a(\xi) v + b(\xi) v^{-1} - \frac{A}{m-2} v^{1-m})^2 }{(a(\xi) v + b(\xi) v^{-1} - \frac{A}{m-2} v^{1-m})^2}.
\end{equation}
If $v(0) = v(L) = r_0$ then $v'(0) = v'(L) = 0$ provided
\begin{equation}\label{ss-9}
  \lambda = - \tfrac{2}{r_0^2} (r_0 - C_0 + \tfrac{A}{m-2} r_0^{2-m} ).
\end{equation}

\end{proof}

\section{Numerical studies}
\label{sec:numerics}
To simulate the fibre coating dynamics and explore beyond the analytical results presented in previous sections, we numerically investigate the problem (\ref{r-1})---(\ref{r-3}). Specifically, we are interested in the structure of energy minimizers and the travelling wave solutions governed by the PDE \eqref{tt-1}.

\begin{figure}
    \centering
    \includegraphics[width = 0.48\textwidth]{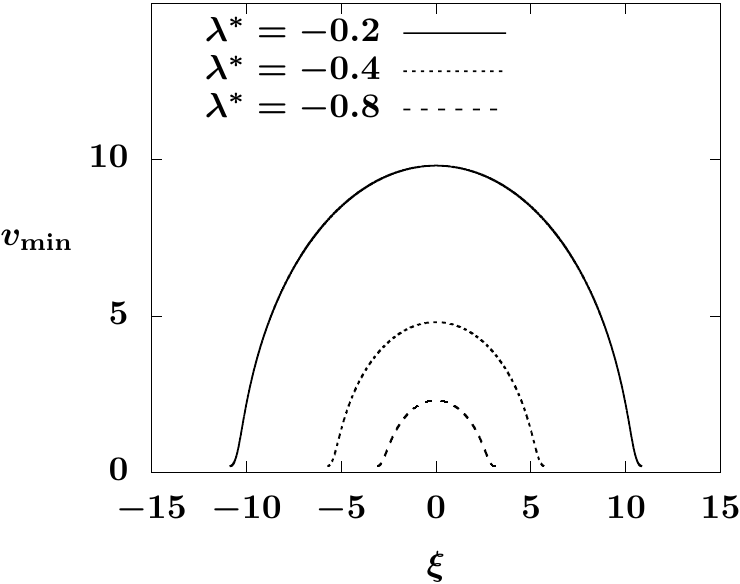}
     \includegraphics[width = 0.48\textwidth]{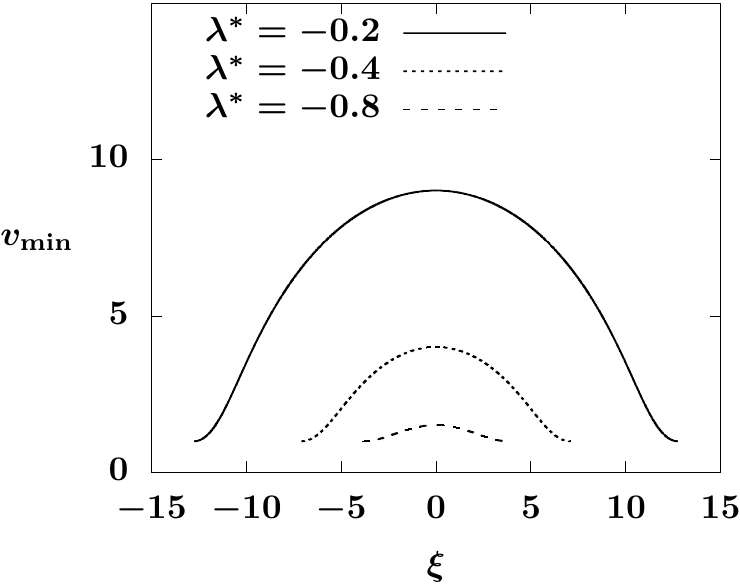}
    \caption{Typical profiles of the energy minimizer $v_{\min}(\xi)$ obtained by numerically solving the ODE \eqref{rt-3} for $A = 0$ and $\lambda = \lambda^*$ for a range of $\lambda^*$ values with (left) $r_0 = 0.2$ and (right) $r_0 = 1$. The minimum and maximum of the profiles are consistent with  \eqref{eq:A=0v1v2} where $C_0 = r_0 + \lambda^* r_0^2/2$ from \eqref{eq:lambstar}. The periods of the solutions are given by \eqref{t-tau}.}
    \label{fig:v_lambdaVarying}
\end{figure}

Firstly, we investigate the energy minimizers $v_{\min}(\xi)$ discussed in Lemma \ref{Lem-1} and their structures. For $\lambda \equiv \lambda^*$ and $A = 0$,  the profile of $v_{\min}(\xi)$ has a unique maximum $\max_{\xi}(v_{\min}) = v_1$ and  minimum $\min_{\xi}(v_{\min}) = v_2$, where $\lambda^*$, $v_1$ and $v_2$ are defined in Lemma \ref{Lem-1} by \eqref{eq:lambstar} and \eqref{eq:A=0v1v2}, respectively. For given values of $\lambda^*$ and $r_0$, the corresponding value of $C_0 = r_0 + \lambda^* r_0^2 / 2$ is obtained based on \eqref{eq:lambstar}, and the period  $\tau$  of the minimizer is obtained from \eqref{t-tau}. Then we numerically solve the ODE \eqref{rt-3} for the solution profile $v_{\min}(\xi)$ on $ -\tau/2 \leqslant \xi \leqslant \tau/2$ subject to the Neumann boundary conditions $v'(\xi) = 0$ at $\xi = \pm \tau/2$. 
When other system parameters are fixed, 
a larger absolute value of $\lambda^*$, $|\lambda^*| = -\lambda^*$, leads to an energy minimizer $v_{\min}(\xi)$ with a smaller magnitude and a smaller period (see Figure \ref{fig:v_lambdaVarying}). The comparison between the two plots with the fibre radius $r_0 = 0.2$ and $r_0 = 1$ in Fig.~\ref{fig:v_lambdaVarying} also shows that with identical $\lambda^*$ values, a larger value of $r_0$ leads to a $v_{\min}(\xi)$ profile with a larger period and smaller spatial variations. This implies that with identical dynamic pressure, droplets on a thick fibre tend to have a smaller bead size than those on a thinner fibre.

\begin{figure}
    \centering
     \includegraphics[width = 0.49\textwidth]{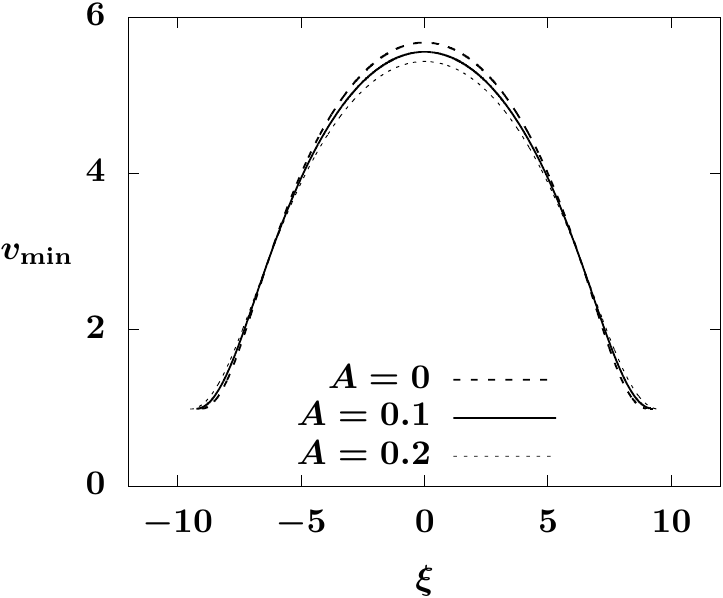}
          \includegraphics[width = 0.49\textwidth]{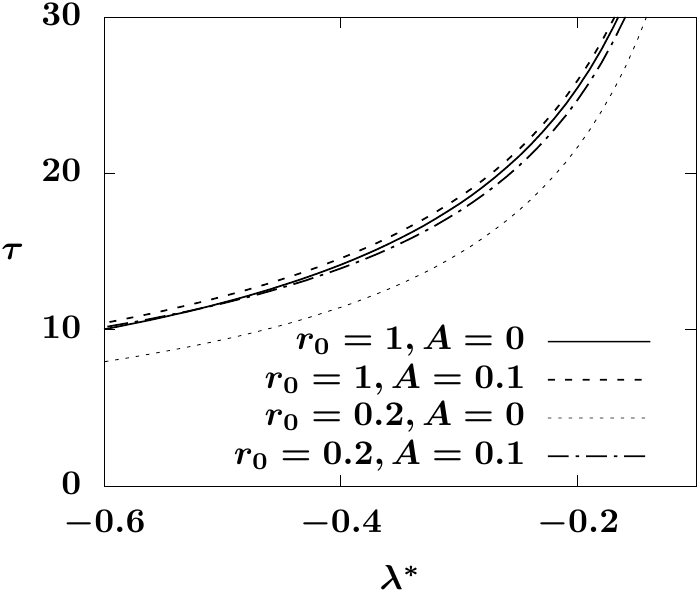}
    \caption{(Left) Profiles of the energy minimizer $v_{\min}(\xi)$ satisfying \eqref{rt-3} for $A = 0, 0.1, 0.2$ with fixed $r_0=1$ and $\lambda=\lambda^* = -0.3$, showing that an increasing $A$ leads to smaller spatial variations. (Right) The period $\tau$ of energy minimizers given by \eqref{t-tau-2} for a varying $\lambda = \lambda^*$ satisfying \eqref{eq:A>0lambda} with different values of $A$ and $r_0$. The parameter $m = 3$ is used in both figures.}
    \label{fig:v_AVarying}
\end{figure}
Figure~\ref{fig:v_AVarying} (left) presents the effects of the stabilization term $A/u^m$, showing that with other system parameters fixed, a positive $A>0$ yields smaller spatial variations in $v_{\min}$. This is also consistent with the observation drawn in \cite{Bert2019}. 
Similar to the $A = 0$ case, the energy minimizers $v_{\min}(\xi)$ for $A > 0$ presented in Fig.~\ref{fig:v_AVarying} (left) are numerically obtained by solving the ODE \eqref{rt-3} for $-\tau/2 \leqslant \xi \leqslant \tau/2$, where the period $\tau$ is given by \eqref{t-tau-2} and the constant $C_0$ is obtained from  \eqref{eq:A>0lambda} with the fibre radius $r_0 = 1$, $\lambda=\lambda^* = -0.3$, and $m = 3$.  With different values of $A$ and $r_0$, Fig.~\ref{fig:v_AVarying} (right) shows a comparison of the period $\tau$ of energy minimizers based on \eqref{t-tau-2} for a varying $\lambda = \lambda^*$ that satisfies \eqref{eq:A>0lambda}. This plot shows that when other system parameters are fixed, larger values of $r_0$ and $A$  typically lead to a larger period $\tau$ for the energy minimizer.

Next we explore the dynamics of the problem (\ref{r-1})---(\ref{r-3}) without gravity effects, that is, without the last term $[Q(u)]_x$. Typical dynamic simulation results of \eqref{r-1} without gravity modulation subject to periodic boundary conditions are plotted in Fig.~\ref{fig:dynamics_noGravity}. The initial data for this simulation is given by $u_0(x) = v(x)$ which satisfies the ODE \eqref{rt-3} with $C_0 = 0.5$, $\lambda = -0.5$ over $0 \leqslant x \leqslant L$ for the period $L = \tau = 10.79$ obtained from \eqref{t-tau}. Other system parameters are $A = 0$, $r_0 = 0.2$, $\sigma = 0.01$, and $M = \int_0^L u^2~dx = 68.6$. We note that the parameter pair $(\lambda, C_0)$ do not satisfy the condition \eqref{eq:A>0lambda}, which implies that the initial profile  $u_0(x)$ is not an energy minimizer of the system. Starting from the initial profile, 
the dynamic solution quickly evolves into a two-hump solution with an additional local maximum at the edge $x = 0$ and $x=L$. Moreover, a local singularity, $u = r_0$, is approached at two global minima near x = $0.9$ and x = $9.9$ where the two hump solutions are connected. An inset plot is also included in Fig.~\ref{fig:dynamics_noGravity} (left) to  display the dynamics near the region where the two humps are connected.
The history of the energy $\mathcal{E}(t)$ in Fig.~\ref{fig:dynamics_noGravity} (right) shows that the energy is dissipating as the solution approaches to the two-hump profile. This confirms the conclusion in Theorem \ref{Th-conv}.
These numerical studies are conducted using fully implicit finite differences with a uniformly-spaced grid, and the numerical scheme is second-order accurate in space and first-order accurate in time. 
In our numerical studies, we also observed cases when a single-hump initial profile can evolve into three or more humps. We suspect that linear stability analysis for steady states of the governing PDE can lead to an estimate of the number of humps and the structure of long-time dynamics. However, a systematic investigation of this problem is beyond the scope of this paper and we may further study this aspect of the problem in the future.
\begin{figure}
    \centering
     \includegraphics[width = 0.49\textwidth]{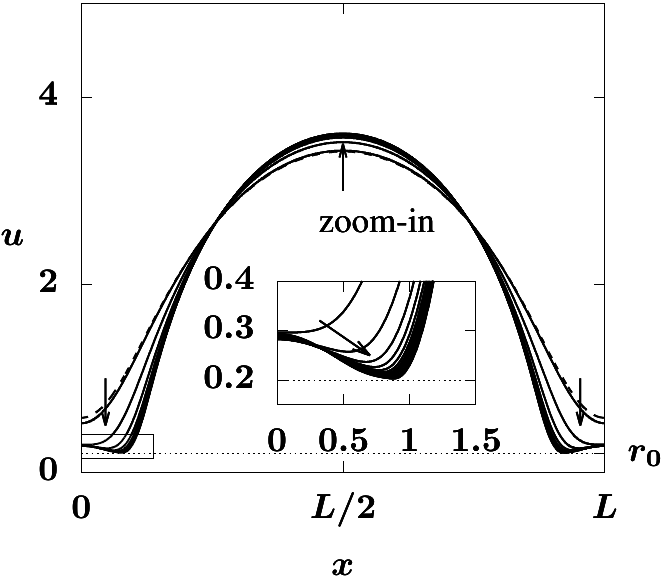}
    \includegraphics[width = 0.49\textwidth]{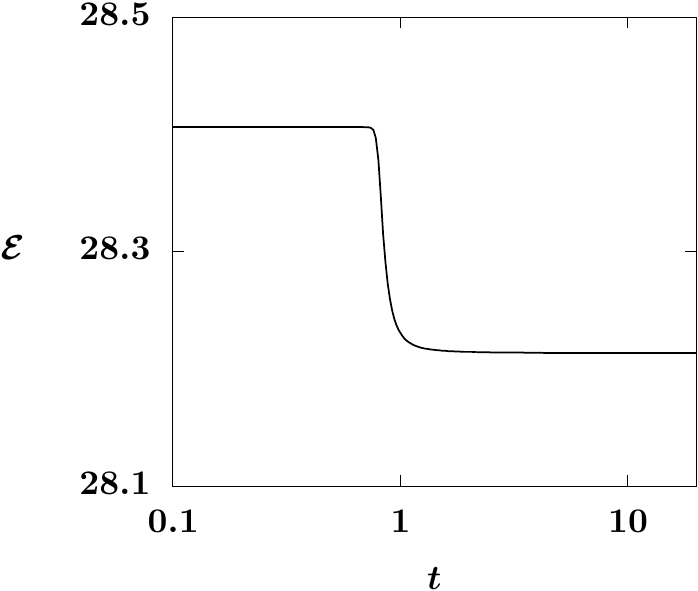}
    \caption{Dynamics of (\ref{r-1})---(\ref{r-3}) without gravity effects for $0 \leqslant t \leqslant 20$ starting from the initial data $u_0(x) = v(x)$ (dashed line) satisfying \eqref{rt-3} with $C_0 = 0.5$, $\lambda = -0.5$. (Left) The solution profile $u(x,t)$ and (right) the energy dissipation $\mathcal{E}(t)$ show that a two-hump solution is reached in long time. The inset plot shows the dynamics near the touch-down point where $u\approx r_0$.
    Other system parameters are $A = 0$, $r_0 = 0.2$, $\sigma = 0.01$.}
    \label{fig:dynamics_noGravity}
\end{figure}

\begin{figure}
    \centering
     \includegraphics[width = 0.48\textwidth]{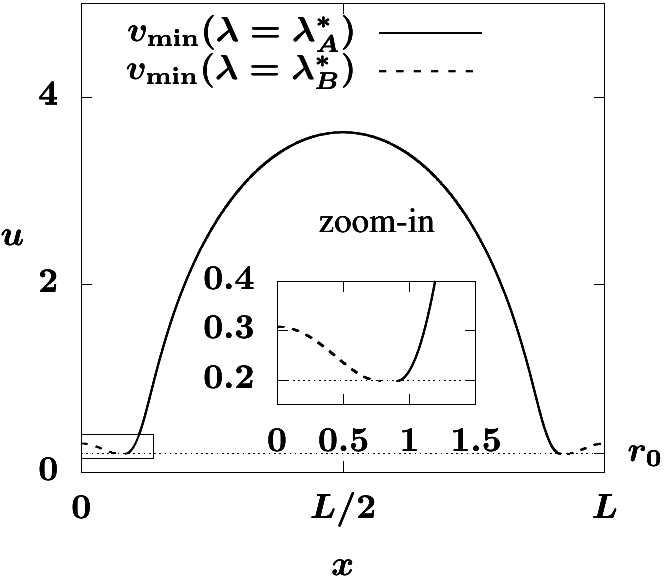}
    \includegraphics[width = 0.48\textwidth]{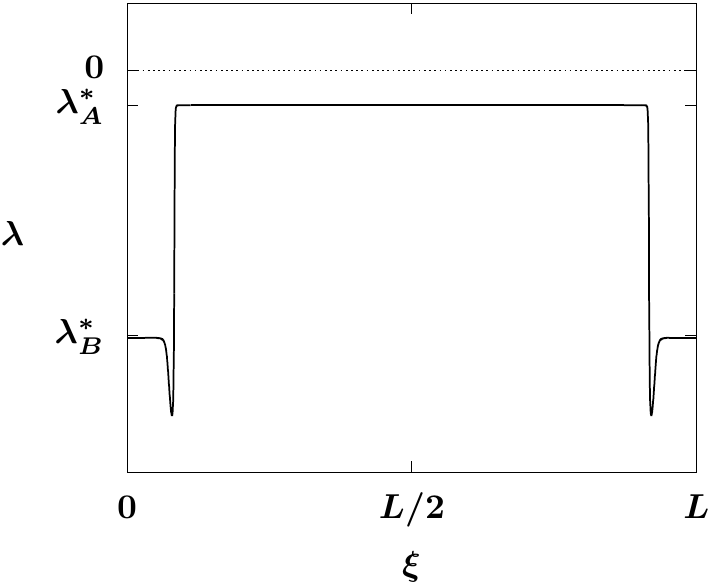}
    \caption{(Left) Energy minimizer $v_{\min}$ and (right) corresponding $\lambda(\xi)$ corresponding to the long-time dynamic solution $u(x)$ at time $t = 20$ in Fig.~\ref{fig:dynamics_noGravity} (left), showing that the PDE solution presented in Fig.~\ref{fig:dynamics_noGravity} approaches to a two-hump energy minimizer characterized by a piece-wise constant function $\lambda(\xi)$. The inset plot shows the two-hump energy minimizer near the touch-down point.}
    \label{fig:lambda_noGravity}
\end{figure}

Figure~\ref{fig:lambda_noGravity}  (right) presents the plot of $\lambda(\xi)$ corresponding to the dynamic solution in Fig.~\ref{fig:dynamics_noGravity} at time $t = 20$.
It shows that $\lambda(\xi)$ approaches a piece-wise constant function where the values of critical $\lambda^*$ are given by $\lambda^*_A = -0.523$ and $\lambda^*_B = -3.955$, respectively. Using these critical $\lambda^*$ values, we numerically solve for the energy minimizers $v_{\min}(\xi)$ derived in Lemma \ref{Lem-1} that correspond to $\lambda = \lambda^*_{A,B}$ and compare them (see Fig.~\ref{fig:lambda_noGravity}  (left)) against the long-time dynamic solution presented in Fig.~\ref{fig:dynamics_noGravity} (left). The two solutions, $v_{\min}(\lambda=\lambda^*_A)$ and $v_{\min}(\lambda=\lambda^*_B)$, are connected at the point where the common minimum value $v = r_0$ is reached (see the inset plot in Fig.~\ref{fig:lambda_noGravity}  (left)).
This comparison shows that the two-hump solution obtained from the direct numerical calculation is in good agreement with the energy minimizer characterized by $v_{\min}$ with a piece-wise constant function $\lambda(\xi) = \lambda^*_{A,B}$.

\begin{figure}
    \centering
     \includegraphics[width = 0.48\textwidth]{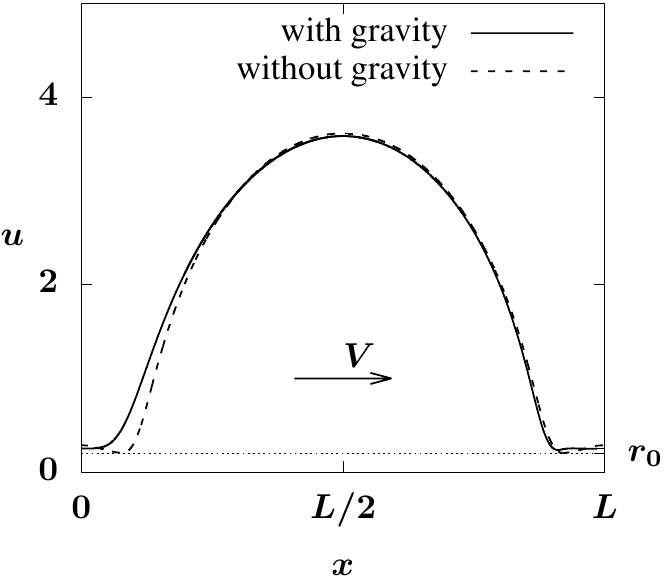}
    \includegraphics[width = 0.48\textwidth]{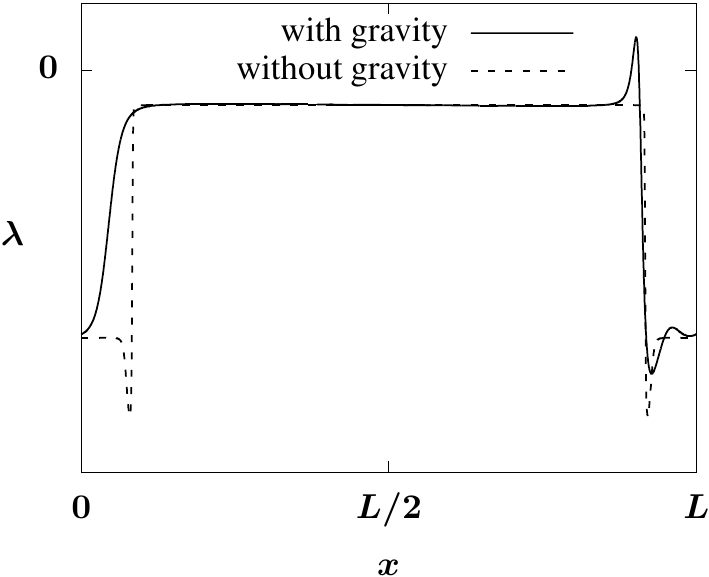}
    \caption{(Left) Solution profiles $u(x)$ and (right) the corresponding $\lambda(x)$ profiles of a travelling wave solution satisfying \eqref{r-000} with gravity effects at a speed $V = 5.12$ compared with long-time dynamic solution shown in Fig.~\ref{fig:dynamics_noGravity} without gravity effects. System parameters are identical to those used in Fig.~\ref{fig:dynamics_noGravity}.}
    \label{fig:profile_withGravity}
\end{figure}

The gravity effects incorporated in (\ref{r-1})---(\ref{r-3}) can lead to a travelling wave solution at a speed $V$ characterized in the travelling wave PDE \eqref{tt-1}. To better understand the influence of gravity effects in the model, we numerically solve the equation \eqref{r-1} starting from the two-hump profile obtained from the simulation shown in Fig.~\ref{fig:dynamics_noGravity} where gravity effects are excluded. The system parameters are also set to be identical to those used in Fig.~\ref{fig:dynamics_noGravity}. The simulation shows that the dynamic solution quickly converges to a travelling wave at a constant speed $V = 5.12$.
A comparison of this travelling wave solution and the two-hump profile without gravity is shown in  Fig.~\ref{fig:profile_withGravity} (left),  where the spatial symmetry around $x = L/2$ in the travelling wave solution is broken due to the presence of gravitational effects. Fig.~\ref{fig:profile_withGravity} (right) depicts the corresponding $\lambda$ profiles for the two solutions. This plot reveals that without gravity the $\lambda$ profile for the two-hump solution approaches a piece-wise constant function,
while the travelling wave solution only keeps the large hump corresponding to $\lambda=\lambda_A^*$ and the small hump with $\lambda=\lambda_B^*$ is saturated by the gravity. In addition, the global minimum of the travelling wave solution is above the fibre radius parameter $r_0 = 0.2$, which indicates that the gravitational effects prevent the singularity observed in the no-gravity case from happening.

Finally, we investigate the convergence to travelling wave solutions with the presence of gravity using the PDE \eqref{tt-1} in the moving reference frame with $\mu = 1$. A travelling wave solution $v_{\min}(\xi)$ at a propagating speed $V$ is a steady state of PDE \eqref{tt-1} and satisfies the fourth-order ODE \eqref{eq:tws_ODE} for $v(\xi)$ on a periodic domain $0 \leqslant \xi \leqslant L$,
\begin{equation}\label{eq:tws_ODE}
\sigma^{-1} [ Q(v) \bigl( (\Phi'(v_{\xi}))_{\xi} -  v^{-1} f(v_{\xi}) +   A v^{-m}  \bigr)_{\xi} ]_{\xi}
+ (Q(v) - \tfrac{V}{2} v^2)_{\xi} = 0.
\end{equation}
The profile of the travelling wave is determined by the period $L$ and the mass $M = \int_0^L v^2~d\xi$. For a given $(L, M)$ pair, we numerically solve the travelling wave ODE \eqref{eq:tws_ODE} for $v_{\min}(\xi)$ and the speed $V$ simultaneously on a periodic domain $0\leqslant \xi \leqslant L$  using finite differences. 
A continuous family of travelling wave solutions to \eqref{eq:tws_ODE} parametrized by $(L, M)$ can be identified.
Using a branch continuation method, we numerically track the travelling waves as the parameter $L$ or $M$ changes. Similar parametric and pseudo-arclength continuation methods are commonly used for bifurcation analysis \cite{keller1987lectures,ji2019numerical,ji2020steady}.
For $L = 9$ and $M = 57.2$, we obtain a travelling wave solution $v_{\min}(\xi)$ with the associated velocity $V = 4.24$ (see the solid curve in Figure~\ref{fig:dynamics_withGravity} (top left)). This velocity agrees with the analytical expression for $V$ provided in Lemma \ref{Lem-1-stw}. The other system parameters are identical to those used in Fig.~\ref{fig:dynamics_noGravity}.
\begin{figure}
    \centering
     \includegraphics[width = 0.45\textwidth]{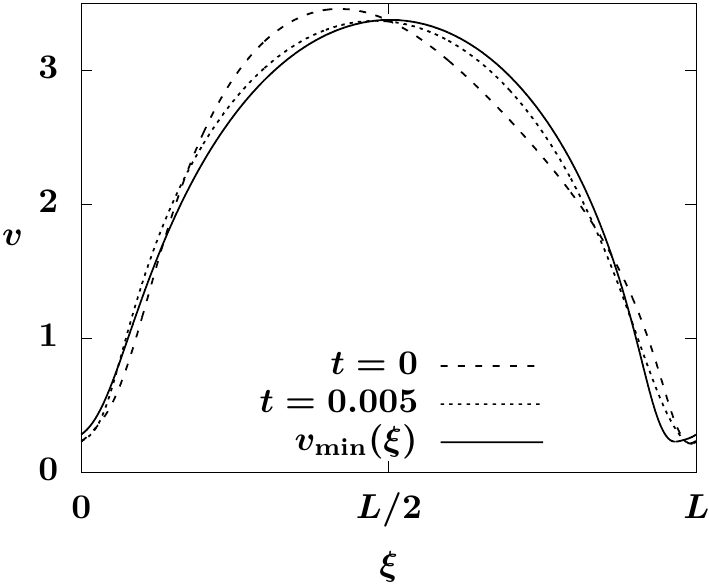}\hspace{0.2in}
         \includegraphics[width = 0.45\textwidth]{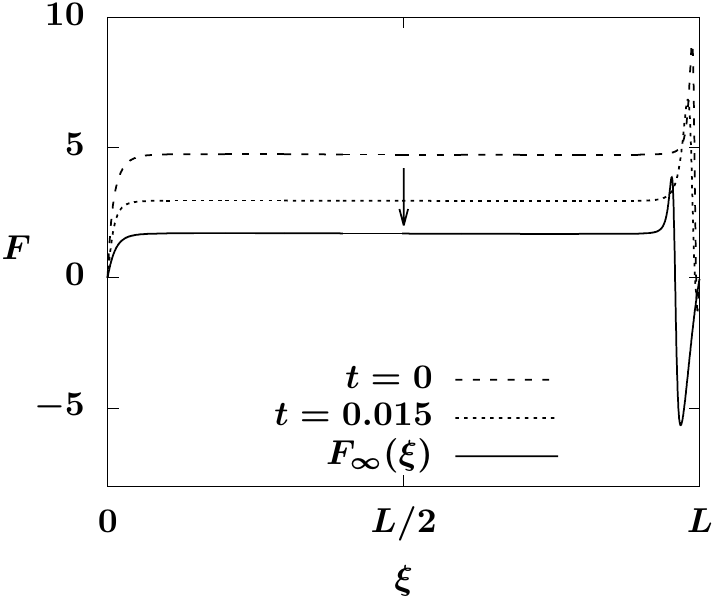}
    \includegraphics[width = 0.45\textwidth]{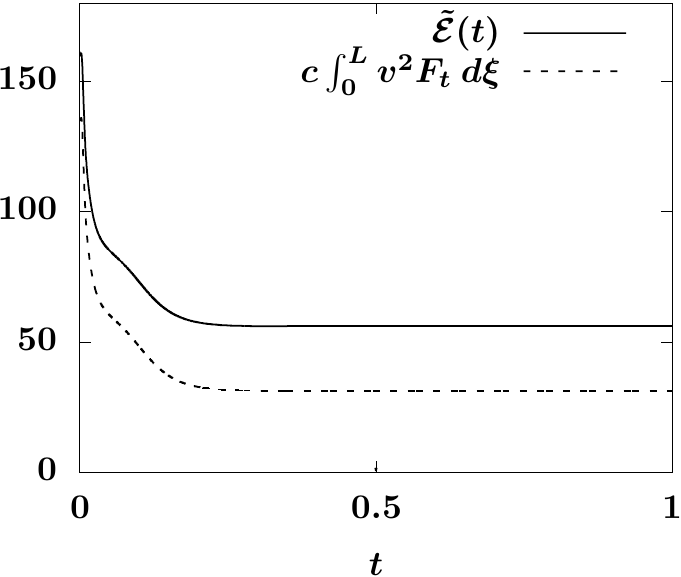}
    \caption{PDE simulation of \eqref{tt-1} with $\mu = 1$ starting from the initial data $v_0(\xi)$ in \eqref{eq:IC} with $\bar{f}=1$, showing (top left)  the convergence of the PDE solution $v(\xi,t)$ to the travelling wave $v_{\min}(\xi)$, (top right) the corresponding evolution of $F(\xi,t)$, and (bottom) the decay of modified energy  $\tilde{\mathcal{E}}$ and $\int_0^L v^2 F_t d\xi$ in time. The travelling wave $v_{\min}(\xi)$ satisfies the ODE \eqref{eq:tws_ODE} and is associated with $L = 9$ and $M = 57.2$. A scaling constant $c = 5\times 10^{-5}$ is used to display both curves $\tilde{\mathcal{E}}$ and $\int_0^L v^2 F_t d\xi$ in a comparable range. Other system parameters are identical to those used in Fig.~\ref{fig:dynamics_noGravity}.}
    \label{fig:dynamics_withGravity}
\end{figure}

Starting from a perturbed initial condition
\begin{equation}
    v_0(\xi) = v_{\min}(\xi)+\bar{\epsilon}\sin(2\pi \bar{f}\xi/L),
\label{eq:IC}
\end{equation}
where $\bar{\epsilon}$ is the magnitude and $\bar{f}/L$ is the frequency of the perturbation,
we numerically solve the travelling wave PDE \eqref{tt-1} and observe that the PDE solution converges to the travelling wave $v_{\min}(\xi)$ as $t\to \infty$. Here we select the magnitude of the perturbation
\begin{equation}
    \bar{\epsilon} = \frac{4}{L}\int_0^Lv_{\min}(\xi)\sin(2\pi \bar{f}\xi/L)~d\xi,
    \label{eq:ic_perturbed}
\end{equation}
such that the initial data satisfies the condition $\int_0^L v_0~d\xi = \int_0^L v^2_{\min}~d\xi = M$ which is necessary for the convergence to the travelling wave $v_{\min}$ due to the conservation of mass condition.
For the simulation shown in Fig.~\ref{fig:dynamics_withGravity}, we set $\bar{f} = 1$ which corresponds to the magnitude $\bar{\epsilon} = 0.51$ based on the formula \eqref{eq:ic_perturbed} for $L = 9$.
To calculate the modified energy $\tilde{\mathcal{E}}(t)$ defined in \eqref{eqn:modified_E}, we specify $\nu(t)$ in the definition of $F(\xi,t)$ in \eqref{eqn:F_def}  based on the formula
\begin{equation}
    \nu = \frac{\frac{V}{2}\int_0^L \frac{v^2(\xi,t)}{Q(v(\xi, t))}~d\xi -L}{\int_0^L\frac{1}{Q(v(\xi, t))}~d\xi}
\end{equation}
so that condition \eqref{ss-11-00} is satisfied. Figure~\ref{fig:dynamics_withGravity} (top right) depicts
the evolution of $F(\xi,t)$ in time as the PDE solution converges to the travelling wave, showing that the condition $F(0,t) = F(L,t) = 0$ in \eqref{ss-11-00} is satisfied.
The plot in Figure~\ref{fig:dynamics_withGravity} (bottom) shows that as the travelling wave $v_{\min}(\xi)$ is approached, both the modified energy $\tilde{\mathcal{E}}(t)$ and the corresponding $\int_0^L v^2 F_t~d\xi$ decay in time, which is consistent with the result shown in Lemma \ref{L-ex-gr}. We note that while the simulations in Fig.~\ref{fig:lambda_noGravity} and Fig.~\ref{fig:dynamics_withGravity} are conducted with $A = 0$, similar dynamics are observed for simulations with small stabilization parameter $A > 0$.

\section{Conclusions}
\label{sec:conclusion}
The main contribution of this paper is showing the existence and long-time behaviour of non-negative weak
solutions for the generalised nonlinear PDE (\ref{r-1})---(\ref{r-3}) using a priori estimates for energy-entropy functionals. Typical numerical studies of the energy functional minimizers and dynamic simulations of the PDE with and without gravitational effects are presented in support of the analytical results. The travelling wave solutions of the model are investigated both analytically and numerically. As $t\to \infty$, with proper system parameters and initial conditions, the solution to \eqref{r-1} converges to a travelling wave solution characterized by \eqref{tt-1}.

\section*{Acknowledgements}
H. Ji was supported by the Simons Foundation Math+X investigator award number 510776. The authors are also grateful to Prof. Andrea L. Bertozzi for helpful discussions.

\section*{Conflict of interest}
None.

\bibliographystyle{abbrv}
\bibliography{PDE_fiber.bib}
\vspace{0.2in}

\end{document}